\documentclass[a4paper, 11pt]{article}
\usepackage{amsmath,amssymb,esint,amscd,xspace,fancyhdr,color,authblk,srcltx,fontenc,bbm}
\usepackage{hyperref}
\usepackage{cite}
\newtheorem{theorem}{Theorem}[section]

\newtheorem{corollary}[theorem]{Corollary}
\newtheorem{definition}[theorem]{Definition}
\newtheorem{lemma}[theorem]{Lemma}

\newtheorem{remark}[theorem]{Remark}
\newenvironment{proof}[1][Proof]{\textbf{#1.} }{\hfill\rule{0.5em}{0.5em}}
{\catcode`\@=11\global\let\AddToReset=\@addtoreset
\AddToReset{equation}{section}

\AddToReset{theorem}{section}

\newtheorem{taggedtheoremx}{Theorem}
\newenvironment{taggedtheorem}[1]
 {\renewcommand\thetaggedtheoremx{#1}\taggedtheoremx}
 {\endtaggedtheoremx}

\title{Revisiting multi-phase variational problems: A Muckenhoupt weight approach}

\author{Thanh-Nhan Nguyen\thanks{Group of Analysis and Applied Mathematics, Department of Mathematics, Ho Chi Minh City University of Education, Ho Chi Minh city, Vietnam; \texttt{nhannt@hcmue.edu.vn}}, Minh-Phuong Tran\footnote{Corresponding author} \thanks{Applied Analysis Research Group, Faculty of Mathematics and Statistics, Ton Duc Thang University, Ho Chi Minh City, Vietnam; \texttt{tranminhphuong@tdtu.edu.vn}}}

\date{\today}

\begin{document}
 
\maketitle
\begin{abstract}

In this paper, we investigate the regularity theory of local minimizers of multi-phase energy functionals. As a key feature of our work, instead of the classical H\"older continuity assumptions on the modulating coefficients and interaction between the growth exponents, we assume that these coefficients belong to a suitable class of Muckenhoupt weights. The presence of multiple growth phases with the degenerate or singular nature of Muckenhoupt weights poses substantial analytical difficulties that prevent a direct application of classical theory. Our approach requires a refinement of localized energy estimates and an adapted iteration scheme that exploits the reverse H\"older properties of the Muckenhoupt weights. As our main results, we establish the higher integrability, local boundedness, and H\"older continuity of local minimizers. Most notably, we prove the Harnack inequality for non-negative local minimizers, which, to the best of our knowledge, stands as the first result of its kind in the multi-phase setting involving Muckenhoupt modulating coefficients. This paper is a contribution toward a better understanding of the qualitative behavior of minimizers in non-uniformly elliptic variational problems, and offers a new framework that complements the existing literature beyond the classical H\"older continuity of modulating coefficients.

\medskip

\noindent 

\medskip



\end{abstract}   
                  
\maketitle
\newpage
\tableofcontents
%-------------------------------------------------

\section{Introduction and main results}
\label{sec:intro}

\subsection{Problem setting} 

In this paper, we are interested in the local minimizers of variational integrals with multiple phases of the form
\begin{align}
\label{eq-main}
W^{1,1}_{\mathrm{loc}}(\Omega) \ni v \mapsto \mathbb{F}(v,\Omega):=\int_\Omega{\left(|\nabla v|^p + \sum_{i=1}^N{a_i(x)|\nabla v|^{p_i}} \right)dx},
\end{align}
where $\Omega \subset \mathbb{R}^n$ is a bounded open subset, with $n \ge 2$, the exponents $1<p \le p_1 \le p_2 \le  \cdots \le p_N<\infty$ and $a_i: \Omega  \to [0,\infty)$ are nonnegative modulating coefficients satisfying
\begin{align}
\label{eq:ai}
a_i \in A_{p_i}, \quad \text{for every } \ i \in I_N:=\{1,2,\cdots,N\}.
\end{align}
Here, $A_{p_i}$ denotes the Muckenhoupt weight class (see Definition \ref{def:Muck}).

The presence of the modulating coefficients $a_i(\cdot)$ allows our energy functional $\mathbb{F}$ to model the energy density of a composite material with multi-phase transitions governed by different growth phases. For instance, in the structure of functional $\mathbb{F}$, one describes a mixture of $N+1$ different materials consisting of a background phase with hardening exponent $p$ together with $N$ phases characterized by the power exponents $p_1, p_2, \dots, p_N$, respectively. The family of modulating coefficients $\{a_i(\cdot): i \in I_N\}$ dictates the geometric distribution of the composite material. Each zero-set $\Omega_i^0 := \{x \in \Omega : a_i(x) = 0\}$ represents the region of structural degeneracy where $p_i$-material component completely vanishes and the energy density governed by the remaining growth exponents. The mathematical formulation of integrals with multiple phases originates from Zhikov’s classic homogenization theory for strongly inhomogeneous materials and elasticity \cite{Zhikov1986, Zhikov1997}. Since the introduction, the double- and multi-phase functionals have attracted considerable attention in the regularity theory of variational integrals with non-standard growth. Furthermore, the study of double- and multi-phase energy functionals is also motivated by their applications to mathematical modeling of physical and mechanical processes, including nonlinear elasticity, electrorheological fluids, and image processing. 

The multi-phase functional of the type \eqref{eq-main} is closely related to the Euler-Lagrange equation
\begin{align*}
-\mathrm{div}\left(p|\nabla u|^{p-2}\nabla u + \sum_{i=1}^N{p_ia_i(x)|\nabla u|^{p_i-2}\nabla u} \right) = 0, \quad \text{in } \ \Omega,
\end{align*}
which exhibits a highly non-uniform and degenerate ellipticity. To see this, let us consider the vector field
\begin{align}\label{eq:HF}
H_\mathbb{F}(x,\xi):=  p|\xi|^{p-2}\xi + \sum_{i=1}^N{p_i a_i(x)}|\xi|^{p_i-2}\xi, \quad x \in \Omega, \ \xi \in \mathbb{R}^n.
\end{align}
The corresponding ellipticity ratio (the ratio between the highest eigenvalue of the Jacobian matrix $\partial_{\xi} H_\mathbb{F}(x,\xi)$ and the lowest eigenvalue governed by the base $p$-growth) is  proportional to $ 1 + \sum_{i=1}^N a_i(x)|\xi|^{p_i-p}$. This ratio, therefore, will blow up whenever the gradient variable $\xi$ tends to infinity in regions where at least one coefficient $a_i(x)$ is positive. In particular, as $|\xi| \to \infty$, the ellipticity ratio grows at the rate of $|\xi|^{p_N-p}$ whenever the leading coefficient $a_N(x)$ is positive. Moreover, the possible vanishing of the modulating coefficients $a_i(\cdot)$ induces transitions between different growth phases, leading to highly non-uniform elliptic structure throughout the domain.

The functional $\mathbb{F}$ belongs to a family of functionals with non-standard growth conditions, initiated by Marcellini \cite{Marcellini1989, Marcellini1991}. A model case that has traditionally attracted a lot of attention over the last few decades is the functional with $(p,q)$-growth, typically given by
\begin{align}\label{eq:mar}
v \mapsto \int_\Omega{f(x,\nabla v) dx},
\end{align}
where $f: \Omega \times \mathbb{R}^n \to \mathbb{R}$ satisfies the unbalanced polynomial growth conditions $|\xi|^p \lesssim f(x,\xi) \lesssim 1+|\xi|^q$ for every $\xi \in \mathbb{R}^n$ and $x \in \Omega$, with $1 < p \le q$. The case $p=q$ corresponds to the classical standard setting (cf. \cite{Min4}). Due to the lack of homogeneity, questions of regularity for minimizers of functional \eqref{eq:mar} driven by non-standard nonlinearities, including double-phase and multi-phase energies as prominent examples, have been the subject of extensive investigation. We refer to \cite{BaCoMin2015, BaCoMin2016, BaCoMin2018, BBO2021, Byun2017Cava, CM2015, CM2015_2, CM2016, FO2019, FM2019, FP2019, FM2021, F2022, FM2023, FM2023_ARMA, FP2024, FMM2004, ELM2004, FS1999} for a selection of remarkable contributions in this direction. 

\subsection{Relevant studies} 
Let us stress for a moment the double-phase model, corresponding to the case $N=1$. In this setting, the functional in \eqref{eq-main} reduces to 
\begin{align}\label{eq:double}
v \mapsto \mathbb{F}_{p,q}(v,\Omega):= \int_\Omega{\left(|\nabla v|^p + a(x)|\nabla v|^q \right)dx}, \quad 1<p \le q,
\end{align}
which exhibits a phase transition in its elliptic behavior dictated by the zero-set $\{x \in \Omega: a(x)=0\}$. The presence of the modulating term $a(x)$ brings several challenges beyond the scope of classical methodologies for the standard $p$-Laplacian. The primary obstruction lies in the loss of scale-invariance and homogeneity, coupled with the potential occurrence of the Lavrentiev phenomenon. In this specific context, a celebrated counterexample of the Lavrentiev phenomenon constructed by Zhikov in \cite{Zhikov1995} exhibits a strict energy gap between the two minimization problems subject to a given non-trivial boundary datum $u_0$, namely, 
$$ 
\inf_{v \in u_0 + W^{1,p}_0(\Omega)} \mathbb{F}_{p,q}(v,\Omega) < \inf_{v \in u_0 + W^{1,p}_0(\Omega) \cap W^{1,\infty}(\Omega)} \mathbb{F}_{p,q}(v,\Omega).
$$
Specifically, Zhikov showed that smooth functions may fail to be dense in the associated non-standard functional spaces, yielding an energy gap where the infimum over this space is strictly less than the infimum over smooth (or Lipschitz) functions. This phenomenon demonstrates that standard approximation may fail, thus invalidating classical analytical tools and posing an obstruction to regularity theory. To avoid this gap, Baroni, Colombo, and Mingione \cite{BaCoMin2015, BaCoMin2016, BaCoMin2018} (see also Colombo and Mingione \cite{CM2015, CM2015_2}) established a sharp necessary and sufficient condition that links the growth imbalance $q-p$ to the H\"older regularity of the coefficients $a(\cdot)$:
\begin{align}\label{eq:pqa}
q \le p+ \frac{\alpha p}{n}, \quad \text{and} \quad a(\cdot) \in C^{0,\alpha}(\Omega), \ \text{for} \ \ \alpha \in (0,1].
\end{align}
This condition, as proved, played a significant role in the regularity theory of double-phase functionals \eqref{eq:double}. Specifically, when \eqref{eq:pqa} holds, local minimizers of \eqref{eq:double} satisfies gradient H\"older continuous, $\nabla u \in C^{0,\beta}_{\text{loc}}(\Omega)$ for some $\beta \in (0,1)$. Afterwards, this regularity condition \eqref{eq:pqa} has inspired numerous subsequent studies in regularity theory for integral functionals regarding double-phase operators and non-uniformly elliptic problems. From pioneering works to recent ones, we list some representative contributions including \cite{Byun2017Cava, CM2016, FM2019, FP2019, FM2021, HOk2019, PNJFA, PNJMAA, BM2020, BO2017, Ok2017}. An extensive list of references can be found in the interesting survey by Mingione and R\u{a}dulescu \cite{MR2021}.

Besides the need to understand the regularity properties and analytical methods for minimizers, a substantial body of work has been devoted to establishing assumptions on various double-phase settings, particularly those concerning the absence of the Lavrentiev phenomenon, the interaction between the growth exponents $p, q$, and the regularity of the coefficient $a(\cdot)$. For instance, the seminal work \cite{CM2015_2} established regularity for bounded minimizers under the H\"older continuity assumption $a \in C^{0,\alpha}$ coupled with the sharp gap condition $q \le p+\alpha$. This framework was subsequently broadened in \cite{BGS2022}, where the authors considered $a \in C^{0,\alpha}(\overline{\Omega})$ alongside the relaxed bound $q \le p+\alpha \max\left\{1,\frac{p}{n}\right\}$. The recent study \cite{Borowski2025} dealt with higher-order H\"older coefficients $a \in C^{k,\alpha}$, yielding the extended gap condition $q \le p+(k+\alpha) \max\left\{1,\frac{p}{n}\right\}$. These references represent only a small portion of the extensive literature devoted to double-phase variational problems. Beyond the setting of H\"older-continuous coefficients, many other significant contributions have focused on more general structures. In this direction, it is also worth highlighting the recent contribution of Adamadze, Diening, Kopaliani, and Ok \cite{ADKO2026}, where a Muckenhoupt-type condition was introduced directly on a class of double-phase integrands.

The multi-phase functional \eqref{eq-main} arises as a natural generalization of the double-phase prototype \eqref{eq:double}. Structural regularity condition \eqref{eq:pqa} was subsequently extended to the triple-phase setting, corresponding to the case $N=2$ in the pioneering work of De Fillipis and Oh \cite{FO2019}. Specifically, to guarantee the absence of a Lavrentiev gap, the authors required a couple of conditions of the form
\begin{align}\label{eq:triple_cond}
\begin{cases}
0 \le a_1(\cdot) \in C^{0,\alpha_1}(\Omega), \quad 0 \le a_2(\cdot) \in C^{0,\alpha_2}(\Omega), \quad \text{for } \ \alpha_1, \alpha_2 \in (0,1],\\[5pt]
p_1 \le p+\displaystyle{\frac{\alpha_1 p}{n}}, \quad p_2 \le p + \displaystyle{\frac{\alpha_2 p}{n}}.
\end{cases}
\end{align}

It is worth emphasizing that the generalization from the double-phase to the triple-phase setting is far from a trivial extension.  Although the model case with three phases was investigated in \cite{FO2019}, one observes that the presence of multiple modulating coefficients gives rise to new analytical difficulties. More precisely, these challenges stem from the nontrivial interaction between different zero-sets of coefficients, where the energy functional may lose its uniform ellipticity due to the possible degeneracy of the modulating coefficients. Under the sharp assumption \eqref{eq:triple_cond}, the authors successfully derived the local $C^{1,\gamma}$-regularity for minimizers of the corresponding triple-phase functional. Moreover, their technique also provides a blueprint that allows one to verify the gradient H\"older continuity results for functionals with an arbitrary finite number of phases. After this work, there is by now an extensive literature on the study of regularity for minimizers of multi-phase functionals under sharp structural conditions regarding the H\"older continuity of modulating coefficients and the growth gaps among the exponents $p,p_1,\cdots,p_N$. More precisely, the sharp structural conditions in the literature typically take the form:
\begin{align}\label{eq:multi_cond}
0 \le a_i(\cdot) \in C^{0,\alpha_i}(\Omega), \ \alpha_i \in (0,1], \quad \text{and} \quad p_i \le p+\displaystyle{\frac{\alpha_i p}{n}}, \quad \text{for all } i \in I_N.
\end{align}
We refer the reader to contributions \cite{F2022, FRZZ2022, BBO2021, FP2024} for a selection of recent developments along this line. 

\subsection{Motivation} 
One observes that the structural constraints imposed on the modulating coefficients and the hardening powers in \eqref{eq:pqa}, \eqref{eq:triple_cond} and \eqref{eq:multi_cond} are natural and optimal for establishing higher regularity, particularly the local gradient H\"older continuity. In this setting, the large exponents $p_i$ are rebalanced by a higher order $\alpha_i$ of H\"older continuity of the corresponding coefficient $a_i(\cdot)$, thereby the interaction between different phases is well-controlled. After the seminal breakthrough by Baroni, Colombo, and Mingione \cite{CM2015, CM2015_2, BaCoMin2015, BaCoMin2018} on double-phase models, the H\"older continuity of $a_i(\cdot)$ becomes standard in later regularity results for multi-phase functionals. The condition \eqref{eq:multi_cond} not only successfully rules out the Lavrentiev gap, but also controls the interaction between different growth phases.  This structural feature is necessary to construct the De Giorgi iteration process, a fundamental step in proving the H\"older continuity of the minimizers. 

At this point, a natural question that may arise here is: in the case when the modulating coefficients are no longer H\"older continuous, namely, \eqref{eq:multi_cond} does not hold, what assumptions should be replaced to guarantee the absence of the Lavrentiev phenomenon? - To address this, let us first revisit the density arguments that play a key role in the seminal proof of \cite[Theorem 4.1]{CM2015}. The standard procedure relies on approximating Sobolev functions by smooth mollifications, whose gradients are pointwise controlled by the Hardy-Littlewood maximal operator $\mathbf{M}$. Consequently, passing to the limit in energy functionals strictly depends on the global boundedness of $\mathbf{M}$ in the corresponding energy space. Motivated by this argument, Muckenhoupt weights arise naturally, since the $A_{p_i}$ classes provide the necessary and sufficient condition for the maximal operator $\mathbf{M}$ to be bounded on weighted Lebesgue spaces, see \cite{Muckenhoupt1972}. Therefore, by treating the modulating coefficients $a_i(\cdot)$ as Muckenhoupt weights in \eqref{eq:ai}, it allows us to apply the Lebesgue Dominated Convergence Theorem to establish the density of smooth functions, and, in turn, rules out the appearance of the Lavrentiev phenomenon. We send the reader to Theorem \ref{theo:Lav} in Section \ref{sec:Lav}. 

In this paper, the primary objective is to address the multi-phase functionals \eqref{eq-main} from a different perspective, under a distinct class of structural conditions imposed on the modulating coefficients $a_i(\cdot)$. As aforementioned, instead of the classical regularity conditions \eqref{eq:multi_cond}, we assume that the coefficients $a_i$ belong to an appropriate Muckenhoupt weight class, as specified in \eqref{eq:ai} (we refer to the subsequent section for the precise definition and notation). By turning from the classical H\"older continuity to Muckenhoupt weight conditions, it enables us to handle a class of irregular modulating coefficients. Indeed, from a physical point of view, in the study of strongly anisotropic materials, the coefficients $a_i$ governing the mixture between different phases may exhibit discontinuities or singularities. Consequently, the H\"older continuity assumptions seem to be restrictive. The Muckenhoupt classes permit the coefficients to vanish or blow up on sets of measure zero by employing an integral condition in a weighted sense, which successfully controls the degenerate or singular elliptic behavior of the functionals. 

A prototypical example of a coefficient $a(\cdot)$ belonging to some admissible $A_p$ class, with the lack of H\"older continuity, is the weight $a(x) = |x|^{-\kappa}$ defined on $\mathbb{R}^n$. In harmonic analysis, it is known that $a(\cdot) \in A_p$ if and only if the exponent satisfies the sharp range $-n(p-1) < \kappa < n$; see, for instance, \cite{55Gra}. Moreover, when $0<\kappa<n$, this function blows up at the origin and is therefore not H\"older continuous; nevertheless, controlled under the Muckenhoupt setting. Besides these classical power-type weights, one can construct Muckenhoupt modulating coefficients with logarithmic singularities. Indeed, given $q > 1$ and $\beta > 1-q$, with $\beta \neq 0$, we  consider the weight function $a: \mathbb{R}^n \to [0, \infty)$ defined by
\begin{align*}
a(x) = \left( \chi_{\mathbb{R}^n \setminus B_{1/e}}(x) + \chi_{B_{1/e}}(x) \log\left(\frac{1}{|x|}\right) \right)^\beta, \quad x \in \mathbb{R}^n.
\end{align*}
It is known that logarithmic weights of this type belong to the Muckenhoupt class $A_q$ whenever $\beta > 1-q$. However, when $\beta>0$, this weight exhibits a logarithmic blow-up at the origin, and therefore fails to be H\"older continuous.

These above examples illustrate that the analysis of multi-phase functionals with non-H\"older continuous modulating coefficients is not only natural but also of independent interest.

\subsection{Main results and Strategy} 
The present paper is one of the contributions towards regularity theory for multi-phase functionals. Aiming at a deeper understanding of the behavior of minimizers to variational problems with non-standard growth, we provide a new structural setting to derive regularity of minimizers through independent proofs and analytical arguments. 

Let us now move our attention to the statements of our main result in this paper. Here and in the rest of the paper, we shall use the notation
\begin{align}\label{def-Hxt}
\mathcal{H}(x,t) := t^p + \sum_{i=1}^N a_i(x) t^{p_i}, \quad \text{for} \  (x,t) \in \Omega \times [0,\infty).
\end{align}
Under the initial assumption \eqref{eq:ai}, the local minimizer of $\mathbb{F}$ can be defined as follows.
\begin{definition}[Local minimizer]
\label{def:local_min}
A function $u \in W^{1,\mathcal{H}}(\Omega)$ is said to be a local minimizer of the functional $\mathbb{F}$ defined in \eqref{eq-main} if it satisfies the minimality condition
\begin{align}\label{var-form}
\mathbb{F}(u,\mathrm{supp}(u-w)) \le \mathbb{F}(w,\mathrm{supp}(u-w)),
\end{align}
for any $w \in W_{\mathrm{loc}}^{1,\mathcal{H}}(\Omega)$ such that $\mathrm{supp}(u-w) \Subset \Omega$.
\end{definition}
In what follows, our primary objective is to derive a series of a priori estimates for local minimizers of the multi-phase functional $\mathbb{F}$.  For the sake of brevity, we will hereafter refer to local minimizers as minimizers throughout the paper. On the other hand, the constants in the estimates depend upon some critical parameters assigned in the problem. We shall summarize the dependence on ``structural data of the problem'', as follows:
\begin{align*}
\mathtt{data} := \left\{ n, N, \{p_i\}_{i=1}^N, \{[a_i]_{A_{p_i}}\}_{i=1}^N \right\}.
\end{align*}

Before stating the main theorems, let us briefly outline the regularity results established, as well as discuss the technical roadmap employed in this paper.  Although our approach draws inspiration from the pioneering method developed for double-phase problems, the structure of the multi-phase functionals coupled with Muckenhoupt weights requires a nontrivial reconstruction of the classical arguments. Our regularity roadmap is structured as follows. We first establish the higher integrability of the gradient, followed by the local boundedness ($L^\infty$-regularity) of the minimizers of multi-phase functionals \eqref{eq-main}. In the next stage, building upon these estimates, we derive the local H\"older continuity of the minimizers by pointing out the decay of their oscillation. This process also involves proving a specialized level-set inequality. Finally, we deduce the Harnack inequality for nonnegative minimizers. This result is of independent interest and provides further quantitative control over the oscillation of the minimizers. 

As a crucial first step, Theorem \ref{theo:higher_int} establishes the local higher integrability of the gradients, which reveals the self-improving property of the multi-phase energy under the structural assumption \eqref{eq:ai}. 
\begin{taggedtheorem}{A}
\label{theo:higher_int}
Let $u \in W^{1,\mathcal{H}}(\Omega)$ be a local minimizer of the multi-phase functional $\mathbb{F}$ defined in \eqref{eq-main} under the assumption \eqref{eq:ai}. Then, $\mathcal{H}(\cdot, |\nabla u|) \in L^{1+\delta_0}_{\mathrm{loc}}(\Omega)$ for some integrability exponent $\delta_0 = \delta_0(\mathtt{data}) \in (0,1)$. More precisely, there exists a constant $C = C(\mathtt{data}) > 0$ such that the following inequality
\begin{equation}\label{eq:gehring}
\left( \fint_{B_{r}} \big[\mathcal{H}(x, |\nabla u|)\big]^{1+\delta_0}  dx \right)^{\frac{1}{1+\delta_0}} \le C \fint_{B_{2r}} \mathcal{H}(x, |\nabla u|)  dx
\end{equation}
holds for every ball $B_r$ such that $B_{2r} \Subset \Omega$.
\end{taggedtheorem}
With respect to the bound \eqref{eq:gehring}, it shows that the gradients of local minimizers in $W^{1,\mathcal{H}}(\Omega)$ enjoy a slightly higher integrability under the constraint \eqref{eq:ai}. We will provide the detailed proof of this theorem in Section \ref{sec:high_int}. In essence, the steps of the proof are inspired by the classical Gehring-type arguments, nontrivially adapted to the multi-phase settings with possibly singular or degenerate Muckenhoupt weights $a_i(\cdot)$. Building upon this result, we subsequently establish the local boundedness of the minimizers in  Theorem \ref{theo-Linf}, whose proof relies on Caccioppoli-type estimates and a subtle adaptation of the celebrated De Giorgi iterative scheme.

\begin{taggedtheorem}{B}
\label{theo-Linf}
Let $u \in W^{1,\mathcal{H}}(\Omega)$ be a local minimizer of the multi-phase functional $\mathbb{F}$ defined in \eqref{eq-main} under the assumption \eqref{eq:ai}. Then, $u \in L_{\mathrm{loc}}^\infty(\Omega)$. Specifically, for any ball $B:=B_r$ such that $2B= B_{2r} \Subset \Omega$, there exists a constant $C = C(\mathtt{data})>0$ such that the following estimate holds:
\begin{align}\label{Linf-u}
\fint_{2B}\mathcal{H}\left(x,\frac{\|u - (u)_{2B}\|_{L^{\infty}(B)}}{r}\right)dx \le C \fint_{2B} \mathcal{H}\left(x,\frac{|u - (u)_{2B}|}{r}\right)dx.
\end{align}
Moreover, if the lower level-set condition $|\{x \in 2B: \, u(x) = 0\}| \ge \sigma_0 |B|$ is satisfied for some $\sigma_0 \in (0, 1)$. Then, there holds
\begin{align}\label{Linf-u-new}
\fint_{2B}\mathcal{H}\left(x,\frac{\|u\|_{L^{\infty}(B)}}{r}\right)dx \le \widetilde{C} \fint_{2B} \mathcal{H}\left(x,\frac{|u|}{r}\right)dx.
\end{align}
Here, $\widetilde{C}$ represents a positive constant depending only on $\mathtt{data}, \sigma_0$.
\end{taggedtheorem}

The estimates \eqref{Linf-u} and \eqref{Linf-u-new} in Theorem \ref{theo-Linf} provide the local $L^\infty$-bounds of minimizers. Specifically, \eqref{Linf-u} yields a quantitative control of the local oscillation in terms of the multi-phase energy density $\mathcal{H}(x,\cdot)$, thereby deriving an a priori $L^\infty$-estimate. In this sense, the local oscillation of minimizers is bounded by the local energy. Moreover, under the additional measure condition on the level set, estimate \eqref{Linf-u-new} removes the dependence on the local average value $(u)_{2B}$, leading to a more intrinsic bound that will play an important role in the derivation of the subsequent level-set inequalities and continuity estimates. Further, the local boundedness of minimizers follows as an immediate consequence, which is formally stated in Corollary \ref{coro-Linf}. The proofs of these results are accomplished in Section \ref{sec:Linf}.

\begin{corollary}\label{coro-Linf}
Under assumptions of Theorem \ref{theo-Linf}, the  following estimate holds true:
\begin{align}\label{Linf-u-coro}
\|u\|_{L^{\infty}(B)} \le C r\left[\fint_{2B} \mathcal{H}\left(x,\frac{|u - (u)_{2B}|}{r}\right)dx\right]^{\frac{1}{p}} + (u)_{2B},
\end{align}
for every ball $B:=B_r$ such that $2B:=B_{2r} \Subset \Omega$.
\end{corollary}

Arising from the local boundedness result established in Theorem \ref{theo-Linf}, in the next stage, we address the local H\"older continuity of minimizers via Theorem \ref{theo:Holder-cont}. 
\begin{taggedtheorem}{C}
\label{theo:Holder-cont}
Let $u \in W^{1,\mathcal{H}}(\Omega)$ be a local minimizer of the multi-phase functional $\mathbb{F}$ defined in \eqref{eq-main} under the assumption \eqref{eq:ai}. Then, there exists $\alpha=\alpha(\mathtt{data}) \in (0,1)$ such that $u \in C^{0, \alpha}_{\mathrm{loc}}(\Omega)$.
\end{taggedtheorem}

Finally, we turn our attention to the Harnack inequality for non-negative minimizers of multi-phase functional \eqref{eq-main}, as stated in Theorem \ref{theo:Harnack}. By asserting that the local supremum of a non-negative minimizer is controlled by its local infimum up to a constant, the estimate \eqref{ineq-Harnack} provides a quantitative tool to control the local behavior of minimizers. 
\begin{taggedtheorem}{D}
\label{theo:Harnack}
Let $u \in W_{\mathrm{loc}}^{1,\mathcal{H}}(\Omega)\cap L_{\mathrm{loc}}^{\infty}(\Omega)$ be a non-negative minimizer of functional $\mathbb{F}$ defined in \eqref{eq-main} under the assumption \eqref{eq:ai}. Then, for every ball $B_{10r} \equiv B_{10r}(x_0) \subset \Omega$, there exists a constant $C=C(\mathtt{data}) \ge 1$ such that 
\begin{align}
\label{ineq-Harnack}
\sup_{B_r} u \le C \inf_{B_r} u.
\end{align}
\end{taggedtheorem}

The proof is in Section \ref{sec:Harnack}. It is also worth emphasizing here that establishing the Harnack inequality \eqref{ineq-Harnack} is highly non-trivial in our setting due to the presence of possibly irregular modulating coefficients. More precisely, by exploiting the weighted structure induced by the Muckenhoupt coefficients, the proof relies on a delicate combination of two independent one-sided estimates: the local upper bound controlling the supremum, and the weak Harnack inequality governing the infimum. Lastly, by carefully bridging these two complementary bounds through a scaling argument, we deduce the full Harnack inequality \eqref{ineq-Harnack}.

The layout of this paper is organized as follows. In Section \ref{sec:pre}, we collect some notations, basic definitions, as well as auxiliary tools that will be employed in the rest of the paper. Section \ref{sec:Lav} is devoted to establishing the absence of the Lavrentiev phenomenon for multi-phase energy functionals $\mathbb{F}$ featuring the Muckenhoupt modulating coefficients $a_i(\cdot)$. In this section, we also prove the boundedness of the Hardy-Littlewood maximal operator on weighted Lebesgue spaces, which ensures the convergence of the approximation scheme via smooth mollifiers. The proofs of the higher integrability (Theorem \ref{theo:higher_int}) and local boundedness results (Theorem \ref{theo-Linf}, along with its associated corollary) are detailed in Sections \ref{sec:high_int} and \ref{sec:Linf}, respectively. Next, in Section \ref{sec:holder_cont}, we derive the crucial De Giorgi-type level-set inequalities, which provide the key ingredient for establishing our main result regarding the local H\"older continuity of minimizers, as stated in Theorem \ref{theo:Holder-cont}. Section \ref{sec:Harnack} concludes the paper by providing the proof of the Harnack inequality asserted in Theorem \ref{theo:Harnack}. 

\section{Background materials}
\label{sec:pre}

In this section, we introduce our notation used throughout the paper, collect the necessary background on relevant function spaces, and recall some auxiliary lemmas that will be employed in the subsequent proofs.

\subsection{Standard notation} 
In what follows, we always use $C$ to denote a positive constant, whose exact value may differ from line to line. Specific occurrences will be denoted by $C_1, C_2, \tilde{C}$, or likewise. For the sake of readability, dependencies of the constants will be explicitly stated at the end of chains of estimates. For instance,  by writing $C(\cdot)$, we emphasize that the constant depends solely on the relevant parameters listed in the parentheses. We denote by $B_R(y_0)$ the open ball in $\mathbb{R}^n$ centered at $y_0 \in \mathbb{R}^n$ with radius $R>0$. When clear by the context, we will often leave out the center of the ball, just writing $B_R \equiv B_R(y_0)$. For a measurable subset $E \subset \mathbb{R}^n$, we denote by $|E|$ the $n$-dimensional Lebesgue measure of $E$. For a given measurable subset $E \subset \mathbb{R}^n$ satisfying $0<|E|<\infty$ and $f: E \to \mathbb{R}$ a locally integrable function, we define its integral average over $E$ by
\begin{align*}
(f)_E: = \fint_{E}{f(x) dx} = \frac{1}{|E|} \int_{E}{f(x) dx}.
\end{align*}

\subsection{Basic definitions and properties} 
Next, let us make a few definitions and properties regarding the Muckenhoupt weights, the reverse H\"older class, Riesz potentials, and Hardy-Littlewood maximal operators, all of which will play a crucial role in our proofs. 
\begin{definition}[Muckenhoupt weights]
\label{def:Muck}
Let $\gamma > 1$ and $\omega \in L^1_{\mathrm{loc}}(\mathbb{R}^n)$ be a non-negative weight. We say that $\omega$ belongs to the Muckenhoupt class $A_{\gamma}$ if 
\begin{equation}\label{A-gamma}
[\omega]_{A_{\gamma}}:= \sup_{B \subset \mathbb{R}^n} \left(\fint_{B} \omega(x)  dx\right) \left(\fint_{B} \omega(x)^{-\frac{1}{\gamma-1}}  dx\right)^{\gamma-1} < \infty, 
\end{equation}
where the supremum is taken over all balls $B \subset \mathbb{R}^n$. Moreover, we define the class $A_{\infty}$ by
\begin{equation*}
A_{\infty} := \bigcup_{\gamma > 1} A_{\gamma}.
\end{equation*}
\end{definition}
An important feature of the Muckenhoupt classes is the self-improving property, which states that if $\omega \in A_\gamma$, then $\omega \in A_{\gamma-\varepsilon}$, for a sufficiently small $\varepsilon > 0$. In addition, any Muckenhoupt weight also  satisfies a strong doubling condition, namely, there exists a constant $C \ge 1$ such that
\begin{align*}
\omega(2B) \le C \omega(B), \quad \text{for all balls } B \subset \mathbb{R}^n,
\end{align*}
where we denote $\omega(E):= \int_E \omega(x) dx$, for any measurable subset $E$.
\begin{definition}[Reverse H\"older classes]
\label{def:RH}
Let $\lambda > 1$, then a non-negative function $\omega \in L^1_{\mathrm{loc}}(\mathbb{R}^n)$ is said to satisfy the reverse H\"older condition of order $\lambda$, denoted by $\omega \in RH_{\lambda}$, if there exists a constant $C \ge 1$ such that the following estimate holds for all balls $B \subset \mathbb{R}^n$:
\begin{align}
\label{RH-lambda}
\left(\fint_{B} \big[\omega(x)\big]^{\lambda}dx\right)^{\frac{1}{\lambda}} \le C \fint_{B} \omega(x)dx,
\end{align}
where the constant $C$ is independent of $B$.
\end{definition}
\begin{remark}\label{rmk:self-improve}
It is worth noting that the Muckenhoupt and reverse H\"older classes exhibit opposite monotonicity properties with respect to their indices. More precisely, for any $1 < \lambda_1 < \lambda_2$, we have $A_{\lambda_1} \subset A_{\lambda_2}$ and $RH_{\lambda_2} \subset RH_{\lambda_1}$. In addition, reverse H\"older classes enjoy a remarkable self-improving property: if $\omega \in RH_{\lambda}$ for some $\lambda > 1$, then $\omega \in RH_{\lambda + \varepsilon}$ for a sufficiently small $\varepsilon > 0$. This higher integrability property will be repeatedly employed in the weighted estimates developed in the sequel. 
\end{remark}
\begin{remark}
\label{rmk:Muc-RH}
Another remarkable property is that the union of all Muckenhoupt classes coincides precisely with the union of all reverse H\"older classes. More specifically, as a classical result established by Coifman and Fefferman \cite{CF1974}, we have:
$$A_\infty \equiv \bigcup_{\gamma>1} A_{\gamma} = \bigcup_{\lambda>1} RH_{\lambda}.$$
\end{remark}

The interplay between Muckenhoupt weights and reverse H\"older classes yields the refined weighted conditions essential to regularity theory. It is therefore natural to introduce the following two-parameter Muckenhoupt class the following two-parameter Muckenhoupt class, which combines the higher integrability properties of the weight and its dual weight. This condition will repeatedly arise in the subsequent weighted estimates. 
\begin{definition}[Two-parameter Muckenhoupt class]
\label{def:2w_muck}
Let $\gamma, \lambda > 1$ and let $\omega \in L^1_{\mathrm{loc}}(\mathbb{R}^n)$ be a non-negative function. We say that $\omega$ belongs to the two-parameter Muckenhoupt class $A_{\gamma,\lambda}$ if 
\begin{align}\label{A-gamma-lamb}
[\omega]_{A_{\gamma,\lambda}} := \sup_B \left(\fint_B [\omega(x)]^\lambda dx\right)^{\frac{1}{\lambda}} \left(\fint_B [\omega(x)]^{-\frac{\gamma}{\gamma-1}} dx\right)^{\frac{\gamma-1}{\gamma}} < \infty,
\end{align}
where the supremum is taken over all balls $B \subset \mathbb{R}^n$.
\end{definition}
In view of \eqref{A-gamma} and \eqref{A-gamma-lamb}, one can observe that the two-parameter class $A_{\gamma,\lambda}$ is deeply linked to the classical Muckenhoupt weights. Moreover, it is straightforward to verify that $\omega$ belongs to the two-parameter class if and only if its $\lambda$-th power falls into a classical Muckenhoupt class. More precisely, we have the equivalence:
\begin{align}
\omega \in A_{\gamma,\lambda} \ \mbox{ if and only if } \ \omega^{\lambda} \in A_{1+\frac{\lambda(\gamma-1)}{\gamma}}.
\end{align}
This observation allows us to transfer many classical properties of $A_p$-weights directly to the class $A_{\gamma,\lambda}$.

Now, we recall the definitions and basic properties of the Riesz potentials and the Hardy–Littlewood maximal operators. These classical operators play a crucial role in the subsequent sections.
\begin{definition}[Riesz potential]
\label{def:I}
Given $\alpha \in (0,n)$, the Riesz potential of order $\alpha$ of a locally integrable function $f \in L^1_{\mathrm{loc}}(\mathbb{R}^n)$, denoted by $\mathbf{I}_{\alpha}$, is defined by
\begin{align}\label{Ifx}
\mathbf{I}_{\alpha} f(x) = \int_{\mathbb{R}^n} \frac{f(y)}{|x-y|^{n - \alpha}} dy, \ \mbox{ for } x \in \mathbb{R}^n \mbox{ and } f \in L^1_{\mathrm{loc}}(\mathbb{R}^n).
\end{align}
Furthermore, given a ball $B \subset \mathbb{R}^n$, we shall use the notation $\mathbf{I}_{\alpha, B} f := \mathbf{I}_{\alpha} \big(\chi_B f\big)$, which will be referred to as the localized Riesz potential of $f$ on $B$. 
\end{definition}
The following lemma, due to Muckenhoupt and Wheeden, describes the weighted boundedness of $\mathbf{I}_\alpha$ entirely in terms of the two-parameter Muckenhoupt class in Definition \ref{def:2w_muck}. For the detailed proof, we refer the reader to \cite{MW1974}. 
\begin{lemma}[Boundedness of Riesz potential]
\label{lem-Ib}
Let $\omega\in L^1_{\mathrm{loc}}(\mathbb R^n,\mathbb R^+)$ and let $\alpha,\gamma,\lambda$ satisfy the Sobolev relation
\begin{align}\notag
0< \alpha <n, \ 1< \gamma < \frac{n}{\alpha}, \mbox{ and } \ \frac{1}{\lambda} = \frac{1}{\gamma} - \frac{\alpha}{n}.
\end{align}
Then the Riesz potential $\mathbf I_\alpha$ is bounded from
$L^\gamma(\mathbb R^n,\omega^\gamma dx)$ into
$L^\lambda(\mathbb R^n,\omega^\lambda dx)$, namely, there exists a constant $C > 0$, independent of $f$, such that
\begin{align}\label{I-alpha-bound}
\left( \int_{\mathbb{R}^n} \left| \mathbf{I}_\alpha f(x) \right|^\lambda \omega(x)^\lambda  dx \right)^{\frac{1}{\lambda}} \le C \left( \int_{\mathbb{R}^n} \left| f(x)\right|^\gamma \omega(x)^\gamma  dx \right)^{\frac{1}{\gamma}}
\end{align}
for every $f \in L^\gamma(\mathbb R^n,\omega^\gamma dx)$. Moreover, the estimate \eqref{I-alpha-bound} holds if and only if $\omega \in A_{\gamma,\lambda}$.
\end{lemma}
Consequently, the class $A_{\gamma,\lambda}$ as the sharp condition governing the weighted boundedness of the Riesz potential. Next, we recall the definition of the Hardy-Littlewood maximal operator and its classical boundedness properties in the weighted Lebesgue spaces setting. This property will be crucial in our subsequent proofs, particularly in the approximation arguments used to exclude the Lavrentiev phenomenon (see Section \ref{sec:Lav}). It has been extensively studied in the literature; the interested reader may refer to the pioneering work of Muckenhoupt \cite{Muckenhoupt1972} for further details.
\begin{definition}[Hardy-Littlewood operator]\label{def:M}
Let $f\in L^1_{\mathrm{loc}}(\mathbb R^n)$. The Hardy-Littlewood maximal operator $\mathbf{M}$  is defined by
\begin{align}\label{Mfx}
\mathbf{M} f(x) = \sup_{\varrho>0} \fint_{B_\varrho(x)} |f(y)|dy, \ \mbox{ for } x \in \mathbb{R}^n \mbox{ and } f \in L^1_{\mathrm{loc}}(\mathbb{R}^n).
\end{align}
\end{definition}
\begin{lemma}[Boundedness of $\mathbf{M}$]
\label{Mbound}
Let $\omega \in L^1_{\mathrm{loc}}(\mathbb{R}^n)$ be a non-negative weight function and let $\gamma > 1$. Then, the Hardy--Littlewood maximal operator $\mathbf{M}$ is bounded on the weighted Lebesgue space $L^{\gamma}(\mathbb{R}^n, \omega  dx)$ if and only if $\omega \in A_{\gamma}$. More precisely, there exists a constant $C > 0$ such that 
\begin{align*}
\left( \int_{\mathbb{R}^n} |\mathbf{M}f(x)|^\gamma \omega(x) dx \right)^{\frac{1}{\gamma}} \le C \left( \int_{\mathbb{R}^n} |f(x)|^\gamma \omega(x) dx \right)^{\frac{1}{\gamma}},
\end{align*}
for every $f \in L^{\gamma}(\mathbb{R}^n, \omega dx)$.
\end{lemma}

\subsection{Young functions and the generalized $N$-functions} 

For the convenience of the reader, we recall below some fundamental definitions along with useful properties concerning the Young and generalized $N$-functions (``Nice'' functions) required for our analysis. We direct the reader to \cite{HH2019} for a more comprehensive treatment and detailed proofs.

\begin{definition}[Orlicz function]
\label{def:Orlicz}
Let $\psi: [0,\infty) \to [0,\infty)$ be a non-decreasing convex function satisfying $\psi(0)=0$. We say that $\psi$ is an Orlicz function if
\begin{align*}
\lim_{h \nearrow \infty} \frac{\psi(h)}{h} = \infty, \ \mbox{ and } \ \lim_{h \searrow 0^+} \frac{\psi(h)}{h} = 0.
\end{align*}
Furthermore, concerning the growth evolution of $\psi$, we introduce the following conditions:
\begin{itemize}
    \item[(i)] $\psi$ is said to satisfy the $\Delta_2$-condition, denoted by $\psi \in \Delta_2$, if there exists a constant $c_1 > 1$ such that $\psi(2h) \le c_1 \psi(h)$ for all $h \ge 0$.
    \item[(ii)] $\psi$ is said to satisfy the $\nabla_2$-condition, denoted by $\psi \in \nabla_2$, if there exists a constant $c_2 > 1$ such that $\psi(h) \le \frac{1}{2c_2} \psi(c_2h)$ for all $h \ge 0$.
\end{itemize}
We shall write $\psi \in \Delta_2 \cap \nabla_2$ if both conditions are satisfied. 
\end{definition}
The following result details the construction of the Orlicz conjugate, which is helpful in our later treatments, and its proof can be found in~\cite{DHHR2011}.
\begin{lemma} 
\label{lem-Young} 
Let $\psi$ be an Orlicz function as in Definition \ref{def:Orlicz}. The conjugate function of $\psi$, denoted by $\psi^*$, is defined as
\begin{align*}
\psi^*(t):= \sup_{h \ge 0} \big[th - \psi(h)\big], \quad \text{for} \ t \ge 0.
\end{align*}
Then, $\psi^*$ is also an Orlicz function and $(\psi^*)^* = \psi$. Moreover, if $\psi \in \Delta_2 \cap \nabla_2$, then 
\begin{align}\label{G-star}
\psi^*\left(\frac{\psi(h)}{h}\right) \le \psi(h) \le \psi^*\left(\frac{2\psi(h)}{h}\right), \quad \text{for all } h > 0.
\end{align}
In particular, for every $\varepsilon \in (0,1)$, there exists a positive constant $C_{\varepsilon}$ such that the following Young-type inequalities
\begin{align}\label{Young-ineq}
ht \le \varepsilon \psi(h) + C_{\varepsilon} \psi^*(t) \quad \text{and} \quad h \frac{\psi(t)}{t} \le \varepsilon \psi(h) + C_{\varepsilon} \psi(t)
\end{align}
hold for all $h \ge 0$ and $t > 0$.
\end{lemma}

In the study of multi-phase problems, the natural functional framework is provided by generalized Orlicz spaces. Since the energy density $\mathcal{H}(x,t)$ defined in \eqref{def-Hxt} depends on both the spatial variable and the gradient, it is therefore convenient to work with the class of generalized $N$-functions. To this end, let us recall the following standard definition.
\begin{definition}[Generalized $N$-function]
\label{def:Gen-N}
A function $\phi: \mathbb{R}^n \times [0,\infty) \to [0,\infty)$ is called a generalized $N$-function if it satisfies the following conditions:
\begin{itemize}
    \item[(i)] For every $t \ge 0$, the mapping $x \mapsto \phi(x,t)$ is measurable.
    \item[(ii)] For almost every $x \in \mathbb{R}^n$, the mapping $t \mapsto \phi(x,t)$ is convex, continuous, and satisfies the limits:
    \begin{align*}
    \lim_{t \to 0^+} \frac{\phi(x,t)}{t} = 0 \quad \text{and} \quad \lim_{t \to \infty} \frac{\phi(x,t)}{t} = \infty.
    \end{align*}
Moreover, for almost every $x$, $\phi(x, \cdot)$ possesses a right-continuous derivative $\phi'(x, \cdot)$ such that $\phi(x, t) = \int_0^t \phi'(x, r) \, dr$.
\end{itemize}
\end{definition}
One also observes that the multi-phase density $\mathcal{H}(x,t)$ introduced in \eqref{eq-main} is a prototype example of a generalized $N$-function whenever the modulating coefficients $a_i(\cdot)$ are non-negative and locally integrable.

In the setting of multi-phase problems, the appearance of the conjugate function $\phi^*$ of a generalized $N$-function is indispensable. In particular, $\phi^*$ arises when dealing with product terms under non-homogeneous growth conditions via Young-type inequalities. It plays a crucial role in absorbing mixed perturbation terms during the derivation of Caccioppoli-type inequalities, a critical step for controlling the oscillation of the gradient. Moreover, the duality relation between $\phi$ and $\phi^*$ is often the primary strategy in the analysis of multi-phase problems, providing several key inequalities employed in regularity estimates.
\begin{definition}[Conjugate $N$-function]
Let $\phi$ be a generalized $N$-function. We denote by $(\phi')^{-1}(x,\cdot)$ the right-continuous inverse of $t \mapsto \phi'(x,t)$. The conjugate function $\phi^*$ associated with $\phi$ is defined by
\begin{align*}
\phi^*(x,t) = \int_0^t (\phi')^{-1}(x,r) dr, \quad x \in \mathbb{R}^n \ \text{ and } t \ge 0.
\end{align*}
\end{definition}
An equivalent characterization is provided by the generalized Legendre-Fenchel transform, namely
\begin{align*}
\phi^*(x,t) = \sup_{r \ge 0} \big(rt - \phi(x,r)\big), \quad \text{for } x \in \mathbb{R}^n \text{ and } t \ge 0.
\end{align*}
The conjugation operation preserves the class of generalized $N$-functions. Moreover, from these definitions, it readily follows that
\begin{align*}
(\phi^*)' = (\phi')^{-1} \quad \text{and} \quad (\phi^*)^* = \phi.
\end{align*}
\begin{remark}
\label{rem:MO_H}
It is worth noting that the multi-phase energy density considered in this paper provides a prototypical example of a generalized $N$-function. Indeed, the function $\mathcal{H}(x,t)$ introduced in \eqref{def-Hxt} satisfies all conditions in the sense of Definition \ref{def:Gen-N}. In the literature, functions of this type are often referred to as \emph{Musielak–Orlicz functions}. For instance, in a typical triple-phase setting, $\mathcal{H}_3(x,t)$ often takes the form:
\begin{align*}
\mathcal{H}_3(x,t) = t^p + a(x)t^q + b(x)t^s, \quad 1 < p \le q \le s,
\end{align*}
where $a(\cdot), b(\cdot) \ge 0$ are bounded measurable weight functions. The presence of several modulating coefficients, which can be generalized to a finite sequence $\{a_i(\cdot)\}$, allows $\mathcal{H}$ to describe multiple competing growth regimes, thereby seamlessly extending the classical double-phase framework to a fully multi-phase setting.
\end{remark}

\subsection{Function spaces} 

To investigate the minimizers of a multi-phase functional $\mathbb{F}$ defined in \eqref{eq-main}, one requires a more general function space setting. In this regard, the corresponding Musielak-Orlicz space $L^{\mathcal{H}}(\Omega)$ and Sobolev space $W^{1,\mathcal{H}}(\Omega)$ provide the natural energy spaces for the multi-phase functional $\mathbb{F}$. 

Let $\phi: \Omega \times [0,\infty) \to [0,\infty)$ be a Musielak--Orlicz function. The Musielak--Orlicz class associated with $\phi$, denoted by $\mathcal{O}^{\phi}(\Omega)$, is defined as the set of all measurable functions $\mathsf{f}: \Omega \to \mathbb{R}$ whose modular functional is finite. More precisely,
\begin{align*}
\mathcal{O}^{\phi}(\Omega) := \left\{\mathsf{f}: \Omega \to \mathbb{R} \ \text{measurable}: \ \rho_\phi(\mathsf{f}) < \infty \right\},
\end{align*}
where the modular $\rho_\phi(\mathsf{f})$ is given by
$$
\rho_\phi(\mathsf{f}) := \int_\Omega \phi(x,|\mathsf{f}(x)|)  dx.
$$
The corresponding \emph{Musielak-Orlicz space} $L^\phi(\Omega)$ is defined as the linear hull of $\mathcal{O}^{\phi}(\Omega)$. Moreover, this space is also a Banach space when endowed with the following Luxemburg norm
\begin{align*}%\label{eq:Luxnorm}
\|\mathsf{f}\|_{L^\phi(\Omega)} := \inf \left\{ s \in (0,\infty) : \rho_\phi\left(\frac{\mathsf{f}}{s}\right) \le 1 \right\}.
\end{align*}

In this paper, associated with the Musielak-Orlicz function $\phi$, we also focus on the \emph{Musielak-Orlicz-Sobolev space} $W^{1,\phi}(\Omega)$, which provides the natural functional framework for the variational problems considered in this paper. This space is the set of all functions $\mathsf{f} \in L^\phi(\Omega)$ whose distributional gradients belong to $L^\phi(\Omega)$, namely,
\begin{align*}
W^{1,\phi}(\Omega) := \left\{ \mathsf{f} \in L^\phi(\Omega) : |\nabla \mathsf{f}| \in L^\phi(\Omega) \right\}.
\end{align*}
The space $W^{1,\phi}(\Omega)$ becomes a Banach space when equipped with the norm
\begin{align*}
\|\mathsf{f}\|_{W^{1,\phi}(\Omega)} = \|\mathsf{f}\|_{L^{\phi}(\Omega)} + \|\nabla \mathsf{f}\|_{L^{\phi}(\Omega)},
\end{align*}
where, for the sake of brevity, we use the abbreviation  $\|\nabla \mathsf{f}\|_{L^{\phi}(\Omega)} := \| |\nabla \mathsf{f}| \|_{L^{\phi}(\Omega)}$. Furthermore, we denote by $W_0^{1,\phi}(\Omega)$ the closure of $C_0^{\infty}(\Omega)$ with respect to the above norm.

The following lemma provides the generalized version of H\"older's inequality in the Musielak-Orlicz setting and will be repeatedly used throughout the paper. We refer to \cite{HH2019} for textbook treatments.
\begin{lemma} \label{lem-Hold} 
Let $\phi$ be a Musielak–Orlicz function and $\phi^*$ denote its conjugate. If $\mathsf{f} \in L^{\phi}(\Omega)$ and $\mathsf{g} \in L^{\phi^*}(\Omega)$, then $\mathsf{f}\mathsf{g} \in L^1(\Omega)$ and there holds
\begin{align}\label{Hold-ineq}
\int_{\Omega} |\mathsf{f}(x)\mathsf{g}(x)|dx \le C \|\mathsf{f}\|_{L^{\phi}(\Omega)} \|\mathsf{g}\|_{L^{\phi^*}(\Omega)}.
\end{align}
\end{lemma}
It is worth emphasizing that the constant $C$ in \eqref{Hold-ineq} is universal and independent of $\mathsf{f}, \mathsf{g}, \phi$ and $\Omega$. Precisely, one may take $C=2$, which follows directly from the generalized Young inequality and the definition of the Luxemburg norm. 

\subsection{Scaling properties and integral averages}

One of the primary technical difficulties in dealing with multi-phase variational problems stems from the spatial dependence of the modulating coefficients $a_i(\cdot)$. To obtain the localized energy estimates in the subsequent sections, we need to introduce the integral averages of our multi-phase energy, which allow us to separate the spatial oscillations. Such averaged functionals will be repeatedly employed in our localization and scaling arguments. The following fundamental lemma provides the scaling properties of the multi-phase density $\mathcal{H}(x,t)$ and its derivative with respect to $t$, briefly denoted by $\mathcal{H}'(x,t)$. The proof is straightforward and follows immediately from the explicit formula of $\mathcal{H}(x,t)$ and its partial derivative $\mathcal{H}'(x,t)$. Therefore, we shall omit the details here.
\begin{lemma}
\label{lem:scaling_H}
For all $x \in \mathbb{R}^n$ and $\lambda, t \ge 0$, the following inequalities hold:
\begin{itemize}
\item[(i)] $\min\{\lambda^p, \lambda^{p_N}\} \mathcal{H}(x,t) \le \mathcal{H}(x, \lambda t) \le \max\{\lambda^p, \lambda^{p_N}\} \mathcal{H}(x,t)$.
\item[(ii)] $\min\{\lambda^{p-1}, \lambda^{p_N-1}\} \mathcal{H}^{\prime}(x,t) \le \mathcal{H}^{\prime}(x, \lambda t) \le \max\{\lambda^{p-1}, \lambda^{p_N-1}\} \mathcal{H}^{\prime}(x,t)$.
\end{itemize}
\end{lemma}

Next, for each ball $B \subset \Omega$, we introduce the Young function $\Phi_B: [0,\infty) \to [0,\infty)$ defined by
\begin{align} \label{def-PHI}
\Phi_B(t) := t^p + \sum_{i=1}^{N} (a_i)_B t^{p_i},
\end{align}
where $(a_i)_B$ denotes the integral average of the modulating coefficient $a_i$ over $B$. It is worth mentioning that $\Phi_B$ can be viewed as the average of the multi-phase density $\mathcal{H}$ over $B$, that is,
\begin{align*}
\Phi_B(t) = \fint_B \mathcal{H}(x,t) dx, \quad \text{for } t \in [0,\infty).
\end{align*}
Further, since $1 < p \le p_1 \le \dots \le p_N$, the mapping $t \mapsto \Phi_B(t)$ is strictly increasing and continuous, and therefore it admits a continuous inverse function $\Phi_B^{-1}: [0,\infty) \to [0,\infty)$. In the next lemmas, we collect several elementary properties of $\Phi_B$, its conjugate, and inverse that will be needed later. Although the proofs of these properties are straightforward, we include the proofs for the reader's convenience. 
\begin{lemma}
\label{lem:scaling_phi}
For each $B \subset \mathbb{R}^n$ and $\lambda, t \ge 0$, the following inequalities hold:
\begin{itemize}
\item[(i)] $\min\{\lambda^p, \lambda^{p_N}\} \Phi_B(t) \le \Phi_B(\lambda t) \le \max\{\lambda^p, \lambda^{p_N}\} \Phi_B(t)$.
\item[(ii)] $\min\{\lambda^{\frac{p}{p-1}}, \lambda^{\frac{p_N}{p_N-1}}\} \Phi_B^*(t) \le \Phi_B^*(\lambda t) \le \max\{\lambda^{\frac{p}{p-1}}, \lambda^{\frac{p_N}{p_N-1}}\} \Phi_B^*(t)$.
\end{itemize}
\end{lemma}
\begin{proof}
The proof of $(i)$ is a direct consequence of the definition of $\Phi_B$ and follows from the condition $p \le p_i\le p_N$, for $i\in I_N$. Specifically, for any $\lambda > 0$, one has $(\lambda t)^p = \lambda^p t^p$ and $(\lambda t)^{p_i} = \lambda^{p_i} t^{p_i}$ hold for each term. By taking the minimum and maximum of the exponents $\{p, p_1, \dots, p_N\}$, we obtain the scaling bounds for $\Phi_B$. Next, the proof of $(ii)$ follows immediately from $(i)$ and standard duality of Young conjugates. Since the growth of $\Phi_B$ is bounded by the exponents $p$ and $p_N$, the conjugate function $\Phi_B^*$ yields the corresponding growth with H\"older conjugate exponents $p' = \frac{p}{p-1}$ and $p_N' = \frac{p_N}{p_N-1}$, respectively. Thus, the desired estimates for $(ii)$ are a direct consequence of the scaling properties of conjugate Young functions. 
\end{proof}
\begin{lemma} 
\label{lem:scale_comparison}
Let $B \subset \Omega$ be a ball and $\Phi_B$ be the Young function defined in \eqref{def-PHI}. Then, for every constant $C_1 \ge 1$, there exists a constant $C_2=C_2(C_1,p,p_N) \ge 1$ such that for all $t_1, t_2 > 0$, the following implications hold:
\begin{align}
\label{eq:hyp_eng_1}
\Phi_B(t_1) \le C_1 \Phi_B(t_2) \implies t_1 \le C_2 t_2.
\end{align}
and the reverse implication holds, namely, for any constant $C_2 \ge 1$, there exists a constant $C_1 = C_1(C_2, p_N) \ge 1$ such that for all $t_1, t_2 > 0$,
\begin{align} 
\label{eq:hyp_eng_2}
t_1 \le C_2 t_2 \implies \Phi_B(t_1) \le C_1 \Phi_B(t_2).
\end{align}
Furthermore, $\Phi_B$ is strictly increasing, and its inverse  $\Phi_B^{-1}$ satisfies a similar scaling property: for any $C_3 \ge 1$, there exists $C_4 = C_4(C_3, p, p_N) \ge 1$ such that
\begin{align}
\Phi_B^{-1}(s_1) \le C_3 \Phi_B^{-1}(s_2) \implies s_1 \le C_4 s_2,
\end{align}
and vice versa. All the scaling constants $C_j$ are independent of $t_1, t_2, s_1, s_2$ and the ball $B$.
\end{lemma}
\begin{proof}
Clearly, if $t_1 \le C_2 t_2$, then the scaling properties from Lemma \ref{lem:scaling_H} imply that $\Phi_B(t_1) \le \max\left\{C_2^p, C_2^{p_N}\right\} \Phi_B(t_2)$. Therefore, it suffices to establish the forward implication in \eqref{eq:hyp_eng_1}. To this end, let us define the scaling factor $k:= \frac{t_1}{t_2} > 0$, which directly yields the relation $t_1 = k t_2$. By the scaling property of $\Phi_B$ derived in Lemma \ref{lem:scaling_phi}, we readily obtain the following estimate:
$$
\Phi_B(t_1) = \fint_{B} \mathcal{H}(x, t_1) dx \ge \min\left\{k^p, k^{p_N}\right\} \fint_{B} \mathcal{H}(x, t_2) dx = \min\left\{k^p, k^{p_N}\right\} \Phi_B(t_2).
$$
If given $\Phi_B(t_1) \le C_1 \Phi_B(t_2)$, the previous lower bound leads to
$$
\min\left\{k^p, k^{p_N}\right\} \fint_{B} \mathcal{H}(x, t_2)  dx \le C_1 \fint_{B} \mathcal{H}(x, t_2) dx.
$$
Noting that $\Phi_B(t_2) > 0$ for $t_2 > 0$, it enables us to divide both sides by $\fint_{B} \mathcal{H}(x, t_2)  dx$ to get $\min\left\{k^p, k^{p_N}\right\} \le C_1$. This guarantees the uniform bound
$$
k \le \max\left\{C_1^{\frac{1}{p_N}}, C_1^{\frac{1}{p}}\right\}.
$$
Consequently, by setting $C_2 := \max\left\{C_1^{\frac{1}{p_N}}, C_1^{\frac{1}{p}}\right\} \ge 1$, we arrive at $t_1 \le C_2 t_2$, which proves \eqref{eq:hyp_eng_1}. The proof for the inverse function $\Phi_B^{-1}$ follows by a similar argument. The proof is then complete.
\end{proof}

The following corollary is an immediate consequence of Lemmas \ref{lem:scaling_phi} and \ref{lem:scale_comparison}.
\begin{lemma}\label{lem-coro-phi-1}
For each $B \subset \mathbb{R}^n$ and $\lambda, t \ge 0$, the following inequality holds:
\begin{align}\notag
\min\left\{\lambda^{\frac{1}{p}}, \lambda^{\frac{1}{p_N}}\right\} \Phi_B^{-1}(t) \le \Phi_B^{-1}(\lambda t) \le \max\left\{\lambda^{\frac{1}{p}}, \lambda^{\frac{1}{p_N}}\right\} \Phi_B^{-1}(t).
\end{align}
\end{lemma}

\section{Absence of Lavrentiev phenomenon}
\label{sec:Lav}

Before describing our work, we briefly recall the notion of the Lavrentiev phenomenon in the Calculus of Variations. One of the fundamental questions in variational problems concerns the existence of minimizers in the natural energy space associated with the variational integral. The Lavrentiev phenomenon occurs when the infimum of a variational integral over its natural energy space (often a generalized Sobolev space) is strictly smaller than the corresponding infimum in the class of smooth (or Lipschitz) admissible functions. This phenomenon is related to the failure of smooth functions to be dense in the natural energy space.

In the elliptic setting, the double-phase functional \eqref{eq:double} has naturally emerged as a crucial model for investigating the Lavrentiev phenomenon in variational problems with non-standard growth. The analysis of local minimizers is important for establishing conditions under which the Lavrentiev phenomenon is absent. In particular, it is known that when $a(\cdot)$ is H\"older continuous, i.e., $a \in C^{0,\alpha}$ for some $\alpha \in (0,1]$, the sharp condition $\frac{q}{p} \le 1+\frac{\alpha}{n}$ guarantees the density of smooth functions in the natural energy space, and therefore the absence of the Lavrentiev phenomenon. The readers may refer to \cite{ELM2004, CM2015, FMM2004, FM2019} for remarkable contributions. 

From the technical point of view, the absence of the Lavrentiev phenomenon is closely related to the density of smooth functions in the natural energy space associated with the functional. In the double-phase setting, such a density property can be established via the classical mollification techniques, which require the boundedness properties of the Hardy-Littlewood maximal operator $\mathbf{M}$ in the energy space. Specifically, as highlighted by Colombo-Mingione in \cite[Section 4]{CM2015}, by ensuring the boundedness of $\mathbf{M}$, standard mollification procedures become admissible, since they are pointwise controlled by the maximal function. By employing the Dominated Convergence Theorem, this yields the convergence of smooth approximations in the modular topology of the energy space, and therefore, concludes the density of $C^\infty$ functions in the corresponding generalized Sobolev space. As a result, the infimum over the natural energy space coincides with the infimum over smooth admissible functions. Hence, the Lavrentiev phenomenon does not occur.

Motivated by the strategy described above, in the multi-phase setting, under the structural Muckenhoupt assumptions \eqref{eq:ai} on the modulating coefficients $a_i(\cdot)$, a natural approach for proving the absence of the Lavrentiev phenomenon is to first establish the boundedness of the Hardy-Littlewood maximal operator on the associated Musielak-Orlicz space. This step is stated in the following lemma.

\begin{lemma}
\label{lem-Maximal}
Let the modular function $\mathcal{H}$ be defined as in \eqref{def-Hxt}, and assume that $a_i \in A_{p_i}$, for every $i \in I_N$. Then, there exists a constant $C=C(\mathtt{data})>0$ such that 
\begin{align}\label{Mphi}
\int_{\mathbb{R}^n} \mathcal{H}\left(x,\mathbf{M}f(x)\right)dx \le  C \int_{\mathbb{R}^n} \mathcal{H}\left(x,|f(x)|\right)dx,
\end{align}
for every $f \in L^{\mathcal{H}}(\mathbb{R}^n)$. In particular, the Hardy–Littlewood maximal operator $\mathbf{M}$ is bounded on the Musielak–Orlicz space $L^{\mathcal{H}}(\Omega)$. 
\end{lemma}
\begin{proof}
The proof follows directly from the boundedness of $\mathbf{M}$ on the unweighted Lebesgue spaces $L^p(\mathbb{R}^n)$, presented in Lemma \ref{Mbound}. Hence, there exists a constant $C_0=C_0(n,p)>0$ such that 
\begin{align*}
\int_{\mathbb{R}^n} |\mathbf{M}f(x)|^p dx \le C \int_{\mathbb{R}^n} |f(x)|^p dx.
\end{align*}
Similarly, for the weighted spaces, the assumption $a_i \in A_{p_i}$ ensures that the weights belong to the Muckenhoupt class. Similarly, for every $i \in I_N$, there exist constants $C_i=C_i(n,p_i,[a_i]_{A_{p_i}}) > 0$ such that 
\begin{align*}
\int_{\mathbb{R}^n} a_i(x)|\mathbf{M}f(x)|^{p_i} dx \le C_i \int_{\mathbb{R}^n} a_i(x)|f(x)|^{p_i} dx. 
\end{align*}
Combining the above estimates by choosing $C=\max\{C_0, C_1, \cdots, C_N\}>0$ (which depends only on the structural data $\mathtt{data}$), it yields 
\begin{align*}
\int_{\mathbb{R}^n} \mathcal{H}\left(x,\mathbf{M}f(x)\right)dx & =  \int_{\mathbb{R}^n} |\mathbf{M}f(x)|^p dx + \sum_{i=1}^N \int_{\mathbb{R}^n} a_i(x)|\mathbf{M}f(x)|^{p_i} dx \\
& \le C\left(\int_{\mathbb{R}^n} |f(x)|^p dx + \sum_{i=1}^N \int_{\mathbb{R}^n} a_i(x)|f(x)|^{p_i} dx\right),
\end{align*}
which leads to \eqref{Mphi}. This completes the proof.
\end{proof}

Next, the boundedness property of $\mathbf{M}$ on the energy space allows us to use the classical mollification arguments. More precisely, the convergence is carefully constructed locally on sufficiently small balls contained in $\Omega$. Through this localized standard mollification, for any given function $v$, one can construct an approximating sequence $(v_m) \subset W^{1,\infty}$ strongly converging to $v$ in the natural energy space. Furthermore, the maximal function provides an integrable majorant for the mollified gradients. Combining this with the Dominated Convergence Theorem, we guarantee that the corresponding energies converge. As a consequence, smooth (or Lipschitz) functions are dense in the natural energy space, and therefore no Lavrentiev gap occurs. This is the content of the next theorem.

\begin{theorem}
\label{theo:Lav}
Let the multi-phase functional $\mathbb{F}$ associated with $\mathcal{H}$ be defined as in \eqref{def-Hxt}, and assume that $a_i \in A_{p_i}$ for every $i \in I_N$. Then, for  every function $v \in W^{1,1}_{\mathrm{loc}}(\Omega)$ and any pair of balls $B \Subset \widetilde{B} \Subset \Omega$ satisfying $\mathbb{F}(v,\widetilde{B}) < \infty$, there exists a sequence $(v_m)_{m \ge 1} \subset W^{1,\infty}(B)$ such that
\begin{align}\label{Lav-convergece}
v_m \to v \mbox{ strongly in } W^{1,p}(B) \ \ \text{and}\ \ \mathbb{F}(v_m,B) \to \mathbb{F}(v,B), \quad \text{as} \  m \to \infty.
\end{align}
\end{theorem}
\begin{proof}
Suppose that $v \in W^{1,1}_{\mathrm{loc}}(\Omega)$ and consider balls $B = B_r \Subset \widetilde{B} \Subset \Omega$ satisfying $\mathbb{F}(v,\widetilde{B}) < \infty$. Let us take $\varepsilon_0 \in (0, 1)$ sufficiently small such that the $\varepsilon_0$-neighborhood of $B$ is strictly contained in $\tilde{B}$, namely, $B_{r+\varepsilon_0} \Subset \widetilde{B} \Subset \Omega$. 

Let $\varphi \in C_0^\infty(B_1)$ be a standard mollifier, that is,
$\varphi\ge0$, $\int_{\mathbb R^n}\varphi dx=1$, $\varphi$ is radially symmetric and $\varphi$ is non-increasing with respect to $|x|$. For every $\varepsilon \in (0,\varepsilon_0)$, we define
\begin{align}\notag
\varphi_\varepsilon(x) = \varepsilon^{-n} \varphi \left(\varepsilon^{-1}x \right), \quad x \in  B_\varepsilon.
\end{align}
It is easy to verify that $\varphi_\varepsilon \in C_0^\infty(B_\varepsilon)$ and
\begin{align}\notag
\int_{\mathbb{R}^n} \varphi_\varepsilon(x)  dx = 1, \ 0 \le \varphi_\varepsilon \le C(n)\varepsilon^{-n}, \ |\nabla \varphi_\varepsilon| \le C(n)\varepsilon^{-(n+1)}.
\end{align}
Taking a sequence $(\varepsilon_m)_{m \ge 1} \subset (0,\varepsilon_0)$ such that $\varepsilon_m \to 0$ as $m \to \infty$, we denote
$$v_{m} := v \ast \varphi_{\varepsilon_m}, \quad m \ge 1.$$
It is obvious that $(v_m)_{m \ge 1} \subset W^{1,\infty}(B)$ and  $v_m \to v$ strongly in $ W^{1,p}(B)$. 

Let $g := |\nabla v| \chi_{\widetilde{B}}$ denote the extension of $|\nabla v|$ by zero outside $\widetilde{B}$. Since $\varepsilon_m < \varepsilon_0$, for any $x \in B$, the mollification $(|\nabla v| \ast \varphi_{\varepsilon_m})(x)$ for $x \in B$ depends only on values within $\widetilde{B}$. Since $\varphi$ is radial and non-increasing,
the classical maximal-function estimate yields $(|g|*\varphi_{\varepsilon_m})(x) \le \mathbf M(g)(x)$. Therefore, we obtain the pointwise bound with constant $1$:
\begin{align*}
|\nabla v_{m}(x)| \le (|g| \ast \varphi_{\varepsilon_m})(x) \le \mathbf{M}(g)(x) \quad \mbox{for every } x \in B.
\end{align*}
Since $t \mapsto \mathcal{H}(x,t)$ is increasing, this ensures that
\begin{align}\notag
\mathcal{H}\big(x, |\nabla v_{m}(x)|\big) \le \mathcal{H}\big(x, \mathbf{M}(g)(x)\big) \quad \mbox{for every } x \in B.
\end{align}
Thanks to Lemma \ref{lem-Maximal}, one gets $\mathcal H(x,\mathbf{M}(g))\in L^1(B)$. More precisely, there holds
\begin{align*}
\int_B \mathcal{H}\big(x, \mathbf{M}(g)(x)\big) dx \le C \int_{\mathbb{R}^n} \mathcal{H}\big(x, g(x)\big) dx = C \int_{\widetilde{B}} \mathcal{H}\big(x, |\nabla v(x)|\big) dx \le C  \mathbb{F}(v,\widetilde{B}) < \infty.
\end{align*}
Since $v\in W^{1,1}_{\rm loc}(\Omega)$, the standard properties of
mollification imply that $\nabla v_m(x)\to \nabla v(x)$ for almost every  $x \in B$. Therefore, $\mathcal H(x,|\nabla v_m(x)|) \to \mathcal H(x,|\nabla v(x)|)$ for almost every $x\in B$. Applying the Dominated Convergence Theorem, it follows that
\begin{align*}
\int_B \mathcal{H}\big(x, |\nabla v_{m}|\big) dx \to \int_B \mathcal{H}\big(x, |\nabla v|\big) dx.
\end{align*}
This ensures the convergence of the energies $\mathbb{F}(v_m, B) \to \mathbb{F}(v, B)$. The proof is complete.
\end{proof}

\section{Local higher integrability}
\label{sec:high_int}

This section is devoted to the higher integrability of local minimizers for the multi-phase functional $\mathbb{F}$ defined in \eqref{eq-main}.  Our primary objective is to prove Theorem \ref{theo:higher_int} under the assumption that the modulation coefficients belong to suitable Muckenhoupt classes. To this end, we first systematically reconstruct the classical variational estimates adapted to the multi-phase structure, including Jensen-type inequalities, Sobolev-Poincar\'e inequalities, and a suitable Caccioppoli inequality. The proofs rely on a flexible combination of classical H\"older and Sobolev-Poincar\'e inequalities together with the self-improving properties of Muckenhoupt $A_{p_i}$ weights. These ingredients enable us to derive a reverse H\"older-type inequality for local minimizers, from which the desired higher integrability follows by an application of Gehring's lemma.

\begin{lemma}
\label{lem:Re-Holder}
Let the integrand $\mathcal{H}$ be defined as in \eqref{def-Hxt}, and assume that $a_i \in A_{p_i}$ for every $i \in I_N$. Then, there exists a constant $\theta = \theta(\mathtt{data})>1$ such that the function $x\mapsto \mathcal H(x,t)$ belongs to the reverse H\"older class, i.e. $\mathcal{H} \left(\cdot,t\right) \in RH_{\theta}$, for every $t \ge 0$. In particular, there exists a constant $C = C(\mathtt{data})>0$ such that
\begin{align}
\label{phi-eps}
\left( \fint_B \big[\mathcal{H} \left(x,t\right)\big]^{\theta}dx \right)^{\frac{1}{\theta}} \le  C \fint_B \mathcal{H}(x,t)dx,
\end{align}
for any ball $B \subset \mathbb{R}^n$ and $t \ge 0$. 
\end{lemma}
\begin{proof}
As we mentioned in Remark \ref{rmk:Muc-RH}, the Muckenhoupt class $A_{\infty}$ can be characterized as 
$$
A_{\infty} = \bigcup_{\gamma >1} A_{\gamma} = \bigcup_{\lambda>1} RH_{\lambda}.
$$ 
Since $a_i \in A_{p_i} \subset A_\infty$ for all $i \in I_N$, there exist $\theta_i > 1$ such that $a_i \in RH_{\theta_i}$. By choosing $\theta = \min\{\theta_1, \theta_2, ..., \theta_N\} > 1$ and using the inclusion property of reverse H\"older classes (namely, $RH_{q} \subset RH_{p}$ for $1<p \le q$), we ensure that $a_i \in RH_{\theta}$ for all $i \in I_N$. Consequently, there exists a constant $C=C(\mathtt{data})>0$ such that
\begin{align*}
\fint_B [a_i(x)]^{\theta} dx \le C \left(\fint_B a_i(x) dx\right)^{\theta}, \ \mbox{ for all } i \in I_N.
\end{align*}
A straightforward calculation yields
\begin{align*}
\fint_B\big[\mathcal{H} \left(x,t\right)\big]^{\theta} dx & = \fint_B \left(t^p + \sum_{i=1}^N a_i(x) t^{p_i}\right)^{\theta} dx \\
& \le C \fint_B \left(t^p\right)^{\theta} dx + C \sum_{i=1}^N \fint_B \big[a_i(x) t^{p_i}\big]^{\theta} dx \\
& \le C \left(\fint_B t^p dx\right)^{\theta}  + C \sum_{i=1}^N \left(\fint_B a_i(x) t^{p_i} dx\right)^{\theta}   \\
& \le  C \left[\fint_B \mathcal{H}(x,t)dx\right]^{\theta}.
\end{align*}
The proof is complete.
\end{proof}

\subsection{Jensen-type inequalities}

This section is devoted to establishing two technical results regarding Jensen-type inequalities adapted to the multi-phase structure under consideration. These auxiliary estimates play a crucial role in the proofs of the subsequent regularity results.

\begin{lemma}[Jensen-type inequality]
\label{lem-Jensen}
Let the integrand $\mathcal{H}$ be defined as in \eqref{def-Hxt}, and assume that $a_i \in A_{p_i}$ for every $i \in I_N$. Then, there holds
\begin{align}\label{phi-phi}
\fint_B \mathcal{H}\left(x,\fint_B |f(y)|dy\right)dx \le  \max_{i \in I_N}\left\{1; [a_i]_{A_{p_i}}\right\} \fint_B \mathcal{H}\left(x,|f(x)|\right)dx,
\end{align}
for any ball $B \subset \mathbb{R}^n$ and every function $f \in L^{\mathcal{H}}(B)$. Moreover, there exist $\theta=\theta(\mathtt{data})>1$ and a constant $C= C(\mathtt{data}) >0$ such that
\begin{align}\label{Jen-ineq}
\left[ \fint_B \left[ \mathcal{H}\left(x, \fint_B |f(y)| \, dy\right) \right]^{\theta} dx \right]^{\frac{1}{\theta}} \le C \fint_B \mathcal{H}\big(x, |f(x)|\big) dx,
\end{align}
for any ball $B \subset \mathbb{R}^n$.
\end{lemma}
\begin{proof}
For any ball $B \subset \mathbb{R}^n$ and every function $f \in L^{\mathcal{H}}(B)$, let us denote the integral average of $|f|$ as
\begin{align*}
J_B = \fint_B |f(y)|dy.
\end{align*}
For every $i \in I_N$, by applying H\"older's inequality with respect to the weight $a_i$, we obtain the following estimates
\begin{align*}
J_B & \le \left(\fint_B a_i(y)^{-\frac{1}{p_i-1}}dy\right)^{\frac{p_i-1}{p_i}} \left(\fint_B a_i(y) |f(y)|^{p_i}dy\right)^{\frac{1}{p_i}}.
\end{align*}
At this stage, since $a_i \in A_{p_i}$, we infer that
\begin{align*}
\fint_B  a_i(x) \big[J_B\big]^{p_i} dx & \le \left[\fint_B  a_i(x)dx \left(\fint_B a_i(x)^{-\frac{1}{p_i-1}}dx\right)^{p_i-1}\right] \left(\fint_B a_i(x) |f(x)|^{p_i} dx\right) \\
& \le [a_i]_{A_{p_i}} \fint_B a_i(x) |f(x)|^{p_i}dx.
\end{align*}
By the classical Jensen inequality for the unweighted convex mapping $t \mapsto t^p$ (since $p \ge 1$), it is then clear that  $\big[J_B\big]^p \le \fint_B |f(x)|^p dx$. Summing these estimates over all $i \in I_N$, we arrive at
\begin{align*}
\fint_B \mathcal{H}\left(x,J_B\right)dx & = \fint_B \big[J_B\big]^pdx + \sum_{i=1}^N \fint_B a_i(x)\big[J_B\big]^{p_i}dx\\
& \le \fint_B |f(x)|^p dx + \sum_{i=1}^N [a_i]_{A_{p_i}} \fint_B a_i(x) |f(x)|^{p_i}dx \\
& \le \max_{i \in I_N}\left\{1, [a_i]_{A_{p_i}}\right\} \fint_B \mathcal{H}\left(x,|f(x)|\right)dx,
\end{align*}
which leads to \eqref{phi-phi}. It remains to deduce \eqref{Jen-ineq}. To this end, observe that since $J_B$ is a constant independent of $x$, it enables us to apply Lemma \ref{lem:Re-Holder} directly to the function $x \mapsto \mathcal{H}(x, J_B)$. Combining this with the previously established estimate \eqref{phi-phi} directly yields the desired inequality \eqref{Jen-ineq}. The proof is now complete.
\end{proof}

\begin{lemma}[Generalized Jensen-type inequality]
\label{lem:general_Jen}
Let the integrand $\mathcal{H}$ be defined as in \eqref{def-Hxt}, and assume that $a_i \in A_{p_i}$ for every $i \in I_N$. Then, there exists a constant $\theta = \theta(\mathtt{data}) \in (1,p]$ such that $a_i \in A_{p_i/\theta}$ for all $i\in I_N$ and
\begin{align}\label{phi-phi-theta}
\fint_B \mathcal{H}\left(x,\left(\fint_B |f(y)|^{\theta}dy\right)^{\frac{1}{\theta}}\right)dx \le  \max_{i \in I_N}\left\{1, [a_i]_{A_{p_i/\theta}}\right\} \fint_B \mathcal{H}\left(x,|f(x)|\right)dx,
\end{align}
for any ball $B \subset \mathbb{R}^n$ and every function $f \in L^{\mathcal{H}}(B)$. 
\end{lemma}
\begin{proof}
The proof follows along the same technique as that of Lemma \ref{lem-Jensen}. For any ball $B \subset \mathbb{R}^n$ and every function $f \in L^{\mathcal{H}}(B)$, let us denote 
$$J_{\theta,B} = \left(\fint_B |f(y)|^{\theta}dy\right)^{\frac{1}{\theta}},$$
where the constant $\theta > 1$ will be specified later. By virtue of the self-improving property of Muckenhoupt weights, the assumption $a_i \in A_{p_i}$ guarantees the existence of sufficiently small constants $\epsilon_i > 0$ such that $a_i \in A_{p_i-\epsilon_i}$ for every $i \in I_N$. We then fix $\theta$ satisfying
\begin{align}\label{theta-Jen}
1< \theta \le \min_{i \in I_N}\left\{p, \frac{p_i}{p_i-\epsilon_i}\right\}.
\end{align}
Since $\theta \le p$, H\"older's inequality immediately implies that 
\begin{align*}
\fint_B \big[J_{\theta,B}\big]^pdx & \le \fint_B |f(y)|^pdy.
\end{align*}
Next, we establish the estimate for the component involving the weight $a_i(x)$. Applying H\"older's inequality with the conjugate exponents $\frac{p_i}{\theta}$ and $\frac{p_i}{p_i-\theta}$, we have
\begin{align*}
J_{\theta,B} & \le \left(\fint_B a_i(y)^{-\frac{\theta}{p_i-\theta}}dy\right)^{\frac{p_i-\theta}{p_i\theta}} \left(\fint_B a_i(y) |f(y)|^{p_i}dy\right)^{\frac{1}{p_i}},
\end{align*}
which implies that
\begin{align}\label{J-theta-1}
\fint_B  a_i(x) \big[J_{\theta,B}\big]^{p_i} dx & \le \fint_B a_i(x)dx \left(\fint_B a_i(y)^{-\frac{\theta}{p_i-\theta}}dy\right)^{\frac{p_i-\theta}{\theta}} \left(\fint_B a_i(y) |f(y)|^{p_i}dy\right).
\end{align}
With the choice of $\theta$ in \eqref{theta-Jen}, one has $a_i \in A_{p_i-\epsilon_i} \subset A_{p_i/\theta}$, which ensures that the Muckenhoupt characteristic is well-defined and satisfies 
\begin{align*}
\fint_B  a_i(x)dx \left(\fint_B a_i(y)^{-\frac{\theta}{p_i-\theta}}dy\right)^{\frac{p_i-\theta}{\theta}} \le [a_i]_{A_{p_i/\theta}} < \infty.
\end{align*}
Thus, we obtain from \eqref{J-theta-1} that
\begin{align*}
\fint_B  a_i(x) \big[J_{\theta,B}\big]^{p_i} dx  \le [a_i]_{A_{p_i/\theta}} \fint_B a_i(x) |f(x)|^{p_i}dx.
\end{align*}
Taking all the above estimates into account and summing over all $i \in I_N$, we conclude that
\begin{align*}
\fint_B \mathcal{H}\left(x, J_{\theta,B}\right)dx & \le \max_{i \in I_N}\left\{1, [a_i]_{A_{p_i/\theta}}\right\} \fint_B \mathcal{H}\left(x,|f(x)|\right)dx.
\end{align*}
The proof of \eqref{phi-phi-theta} is now complete. 
\end{proof}

\subsection{Sobolev-Poincar\'e inequality}

The next step is to establish the Sobolev- Poincar\'e inequality associated with the multi-phase integrand $\mathcal{H}$. As a preliminary step, we first prove an estimate involving the localized Riesz potential, which is presented in the following lemma. 
\begin{lemma}
\label{lem-Riesz-I1}
Let the integrand $\mathcal{H}$ be defined as in \eqref{def-Hxt}, and assume that $a_i \in A_{p_i}$ for all $i \in I_N$. Then, there exist $\theta = \theta(\mathtt{data}) > 1$ and a constant $C = C(\mathtt{data}) > 0$ such that 
\begin{align}\label{ineq-H-I1}
\left[\fint_{B} \mathcal{H}^{\theta}\left(x,\frac{\mathbf{I}_{\alpha,B}\big(|f|\big)(x)}{r^{\alpha}}\right)dx\right]^{\frac{1}{\theta}} \le C \fint_{B} \mathcal{H}\left(x, |f(x)|\right)dx,
\end{align}
for any ball $B:=B_r \subset \mathbb{R}^n$ and every function $f \in L^{\mathcal{H}}(B)$. Here, $\mathbf{I}_{\alpha,B}$ indicates the locally Riesz potential of order $\alpha \in (0,n)$, as defined in Definition \ref{def:I}.
\end{lemma}
\begin{proof}
Let $B = B_r \subset \Omega$ be a ball of radius $r>0$ and denote $g := \chi_B |f|$. For each $k \in \mathbb{N}_0$, we define the new radii $r_k = 2^{-k}r$. Fix $x \in B$, it is clear that $B \subset B_{2r}(x)$. Thus, we can decompose the domain as
\begin{align*}
B \subset B_{2r}(x) = \bigcup_{k=0}^{\infty} \mathcal{D}_k(x), \mbox{ where } \mathcal{D}_k(x) := B_{2r_k}(x)\setminus B_{r_k}(x).
\end{align*}
For every $y \in \mathcal{D}_k(x)$, one has $|x - y| \ge r_k$. The localized Riesz potential can be estimated by partitioning the integral domain
\begin{align}
\mathbf{I}_{\alpha,B}\big(|f|\big)(x) & = \int_{B} \frac{g(y)}{|x-y|^{n-\alpha}} dy \notag \\
& \le \sum_{k=0}^{\infty} \int_{\mathcal{D}_k(x)} \frac{g(y)}{|x-y|^{n-\alpha}} dy \notag \\
& \le \sum_{k=0}^{\infty} \left(\frac{1}{r_k}\right)^{n-\alpha} \int_{\mathcal{D}_k(x)} g(y) dy \notag \\
& \le C \sum_{k=0}^{\infty} \left(\frac{1}{r_k}\right)^{n-\alpha} \left(2r_k\right)^n  \fint_{B_{2r_k}(x)} g(y) dy \notag \\
&\le C  r^{\alpha} \sum_{k=0}^{\infty} 2^{-k\alpha} \fint_{B_{2r_k}(x)} g(y) dy.\label{II-I1}
\end{align}
Next, it is possible to cover the ball $2B$ by a family of balls $\mathcal{B}^k = \{\widetilde{B}\}$ of radius $r_k$ with bounded overlap such that $2B \subset \bigcup_{\widetilde{B} \in \mathcal{B}^k} \widetilde{B} \subset 4B$. Therefore, \eqref{II-I1} can be discretized as
\begin{align}\notag
\mathbf{I}_{\alpha,B}\big(|f|\big)(x) &\le C  r^{\alpha} \sum_{k=0}^{\infty} 2^{-k\alpha} \sum_{\widetilde{B} \in \mathcal{B}^k} \chi_{\widetilde{B}}(x) \fint_{\widetilde{B}} g(y) dy,
\end{align}
By using the convexity of $\mathcal{H}(x, \cdot)$ via a weighted version of Jensen's inequality for series, it follows that
\begin{align*}
\mathcal{H}\left(x, \frac{\mathbf{I}_{\alpha,B}\big(|f|\big)(x)}{r^{\alpha}} \right) &\le \mathcal{H}\left(x, C \sum_{k=0}^{\infty} 2^{-k\alpha} \sum_{\widetilde{B} \in \mathcal{B}^k} \chi_{\widetilde{B}}(x) \fint_{\widetilde{B}} g(y) dy \right) \\
&\le C \sum_{k=0}^{\infty} 2^{-k\alpha p} \mathcal{H}\left(x, \sum_{\widetilde{B} \in \mathcal{B}^k} \chi_{\widetilde{B}}(x) \fint_{\widetilde{B}} g(y) dy \right).
\end{align*}
Next, applying Minkowski's inequality in the space $L^\theta(B)$ (since $\theta > 1$), one obtains from the previous estimate that
\begin{align}\label{minkowski-step}
& \left[\fint_{B} \mathcal{H}^{\theta}\left(x,\frac{\mathbf{I}_{\alpha,B}\big(|f|\big)(x)}{r^{\alpha}}\right)  dx \right]^{\frac{1}{\theta}} \notag \\
& \qquad \qquad \le C \sum_{k=0}^{\infty} 2^{-k \alpha p} \left[\fint_{B} \mathcal{H}^{\theta}\left(x, \sum_{\widetilde{B} \in \mathcal{B}^k} \chi_{\widetilde{B}}(x) \fint_{\widetilde{B}} g(y) \, dy\right)  dx \right]^{\frac{1}{\theta}}.
\end{align}
Since the family of balls $\mathcal{B}^k$ has bounded overlap, the functions $\chi_{\widetilde{B}}$ have almost disjoint supports for each fixed $k$. This property allows us to distribute the power $\theta$ inside the summation over $\mathcal{B}^k$. Consequently, by the subadditivity of the mapping $t \mapsto t^{1/\theta}$, we arrive at
\begin{align}\label{cover-step}
&\left[\fint_{B} \mathcal{H}^{\theta}\left(x, \sum_{\widetilde{B} \in \mathcal{B}^k} \chi_{\widetilde{B}}(x) \fint_{\widetilde{B}} g(y)  dy\right)  dx \right]^{\frac{1}{\theta}} \notag \\
&\qquad \qquad \le C \left[\sum_{\widetilde{B} \in \mathcal{B}^k} \frac{|\widetilde{B}|}{|B|} \fint_{\widetilde{B}} \mathcal{H}^{\theta}\left(x, \fint_{\widetilde{B}} g(y)  dy\right)  dx \right]^{\frac{1}{\theta}} \notag \\
&\qquad \qquad \le C \left[\frac{|\widetilde{B}|}{|B|}\right]^{\frac{1}{\theta}} \sum_{\widetilde{B} \in \mathcal{B}^k} \left[\fint_{\widetilde{B}} \mathcal{H}^{\theta}\left(x, \fint_{\widetilde{B}} g(y) \, dy\right)  dx \right]^{\frac{1}{\theta}}.
\end{align}
At this stage, since $a_i \in A_{p_i}$ for all $i \in I_N$, the multi-phase integrand satisfies a reverse H\"older-type condition. Thus, for $\theta > 1$ chosen sufficiently close to $1$, we can apply Jensen-type inequality \eqref{Jen-ineq} to deduce that
\begin{align*}
\left[\fint_{\widetilde{B}} \mathcal{H}^{\theta}\left(x, \fint_{\widetilde{B}} g(y) dy\right) dx \right]^{\frac{1}{\theta}} &\le C \fint_{\widetilde{B}} \mathcal{H}\left(x, \fint_{\widetilde{B}} g(y)  dy\right) dx \\
&\le C \fint_{\widetilde{B}} \mathcal{H}\big(x, \chi_B(x) |f(x)|\big)  dx.
\end{align*}
Combining this with \eqref{minkowski-step} and \eqref{cover-step}, and reabsorbing localized averages on $\widetilde{B}$ back to the global average on $B$ via the scaling factor $|B|/|\widetilde{B}|$, we get
\begin{align}
\left[\fint_{B} \mathcal{H}^{\theta}\left(x,\frac{\mathbf{I}_{\alpha,B}\big(|f|\big)(x)}{r^{\alpha}}\right)dx\right]^{\frac{1}{\theta}}  & \le C \sum_{k=0}^{\infty} 2^{-k \alpha p} \left[\frac{|\widetilde{B}|}{|B|}\right]^{\frac{1}{\theta}} \sum_{\widetilde{B} \in \mathcal{B}^k}  \fint_{\widetilde{B}} \mathcal{H}\big(x, \chi_B(x) |f(x)|\big)dx \notag \\
& \le C \sum_{k=0}^{\infty} 2^{-k \alpha p} \left[\frac{|\widetilde{B}|}{|B|}\right]^{\frac{1}{\theta}-1} \fint_{B} \mathcal{H}\left(x, |f(x)|\right)dx \notag \\
& \le C \sum_{k=0}^{\infty} 2^{-k\left(\alpha p+\frac{n}{\theta}-n\right)}  \fint_{B} \mathcal{H}\left(x, |f(x)|\right)dx. \notag %.\label{III_I1}
\end{align}
Summarizing, we obtain the following estimate
\begin{align}\label{final-series-step}
\left[\fint_{B} \mathcal{H}^{\theta}\left(x,\frac{\mathbf{I}_{\alpha,B}\big(|f|\big)(x)}{r^{\alpha}}\right)  dx \right]^{\frac{1}{\theta}} \le C \left( \sum_{k=0}^{\infty} 2^{-k\left(\alpha p + \frac{n}{\theta} - n\right)} \right) \fint_{B} \mathcal{H}\left(x, |f(x)|\right)  dx.
\end{align}
Note that $\alpha p > 0$. By choosing $\theta > 1$ sufficiently close to $1$ such that the structural threshold remains strictly positive, specifically,
\begin{align*}
\alpha p + \frac{n}{\theta} - n > 0.
\end{align*}
This condition guarantees the convergence of our geometric series on the right-hand side of \eqref{final-series-step}. Then, the desired estimate follows immediately from \eqref{final-series-step}. The proof is now complete.
\end{proof}

\begin{lemma}[Sobolev-Poincar\'e inequalities]
\label{lem-Sob-Poin}
Let the integrand $\mathcal{H}$ be defined as in \eqref{def-Hxt}, and assume that $a_i \in A_{p_i}$ for all $i \in I_N$. Consider a ball $B := B_r \subset \Omega$. Then, there exist $\theta=\theta(\mathtt{data})>1$ and a constant $C=C(\mathtt{data})>0$ such that
\begin{align}\label{SP-ineq}
\left[\fint_{B} \mathcal{H}^{\theta}\left(x,\frac{|u-(u)_B|}{r}\right)dx\right]^{\frac{1}{\theta}} \le C \fint_{B} \mathcal{H}\left(x,|\nabla u|\right)dx,
\end{align}
for every $u \in W^{1,\mathcal{H}}(B)$. In particular, if $|E|>0$, where  $E = \{x \in B:  u(x) = 0\}$ then there holds
\begin{align}\label{SP-ineq-new}
\left[\fint_{B} \mathcal{H}^{\theta}\left(x,\frac{|u|}{r}\right)dx\right]^{\frac{1}{\theta}} \le C \left(1 + \frac{|B|}{|E|}\right)^{p_N} \fint_{B} \mathcal{H}\left(x,|\nabla u|\right)dx.
\end{align}
\end{lemma}
\begin{proof}
Consider $B:= B_r \subset \Omega$ and a function $u \in W^{1,\mathcal{H}}(B)$. By the classical Riesz potential estimate, for every $x \in B$, we have
\begin{align}\notag
|u(x) - (u)_B| \le C\int_B \frac{|\nabla  u(y)|}{|x-y|^{n-1}} dy = C \mathbf{I}_{1,B}(|\nabla  u|)(x),
\end{align}
which leads to
\begin{align}\notag
\left[\fint_{B} \mathcal{H}^{\theta}\left(x,\frac{|u-(u)_B|}{r}\right)dx\right]^{\frac{1}{\theta}} \le C \left[\fint_{B} \mathcal{H}^{\theta}\left(x,\frac{\mathbf{I}_{1,B}(|\nabla  u|)(x)}{r}\right)dx\right]^{\frac{1}{\theta}}.
\end{align}
Thanks to Lemma \ref{lem-Riesz-I1}, we deduce \eqref{SP-ineq}. To establish \eqref{SP-ineq-new}, we further assume that $|E| = |\{x \in B: \ u(x) = 0\}|>0$. For every $x \in E$, since $u(x) = 0$, it allows us to write 
$$(u)_B = \fint_B \big[u(y) - u(x)\big] dy,$$ 
which in turn yields that
\begin{align}
|(u)_B| & \le \frac{1}{|E|} \int_E \left( \frac{1}{|B|} \int_B |u(y) - u(x)| dy \right) dx\notag \\
& \le \frac{|B|}{|E|} \fint_B \fint_B |u(y) - u(x)| \,dy \,dx \notag \\
& \le \frac{|B|}{|E|} \fint_B \fint_B \left( |u(y) - (u)_B| + |u(x) - (u)_B| \right) dy \,dx \notag \\
& = 2 \frac{|B|}{|E|} \fint_B |u(y) - (u)_B| dy.\notag
\end{align}
Applying the generalized Jensen inequality \eqref{Jen-ineq} and H\"older's inequality, it follows that
\begin{align}
\left[\fint_B \mathcal{H}^{\theta}\left(x, \frac{|(u)_B|}{r}\right) dx\right]^{\frac{1}{\theta}} & \le C \left(\frac{|B|}{|E|}\right)^{p_N}  \left[\fint_B\mathcal{H}^{\theta}\left(x, \fint_B \frac{|u(y) - (u)_B|}{r} dy \right) dx\right]^{\frac{1}{\theta}} \notag \\
& \le C \left(\frac{|B|}{|E|}\right)^{p_N} \fint_B\mathcal{H}\left(x, \frac{|u - (u)_B|}{r} \right) dx \notag \\
& \le C \left(\frac{|B|}{|E|}\right)^{p_N} \left[\fint_B\mathcal{H}^{\theta}\left(x, \frac{|u - (u)_B|}{r} \right) dx\right]^{\frac{1}{\theta}}. \label{SP-100}
\end{align}
Making use of the following fundamental inequality
\begin{align*}
\mathcal{H}^{\theta}\left(x, \frac{|u|}{r}\right) \le C \left[\mathcal{H}^{\theta}\left(x, \frac{|u-(u)_B|}{r}\right) + \mathcal{H}^{\theta}\left(x, \frac{|(u)_B|}{r}\right)\right],
\end{align*}
one readily obtain \eqref{SP-ineq-new} from \eqref{SP-ineq} and \eqref{SP-100}. This completes the proof.
\end{proof}

\subsection{Caccioppoli-type inequality}

We now turn our attention to deriving the Caccioppoli inequality for local minimizers of the multi-phase functional $\mathbb{F}$, which plays an indispensable role in establishing their subsequent H\"older regularity.

\begin{lemma}[Caccioppoli-type inequality]
\label{lem-Cacci}
Assume that $u \in W^{1,\mathcal{H}}(\Omega)$ is a local minimizer of the functional $\mathbb{F}$ defined as in \eqref{eq-main} with $a_i \in A_{p_i}$ for every $i \in I_N$. Let $0<\varrho < R$ such that $B_R \Subset \Omega$.  Then, there exists a constant $C = C(\mathtt{data})>0$ such that
\begin{align}\label{Cacci-ineq-0}
\int_{B_{\varrho}} \mathcal{H}(x,|\nabla u|) dx \le C \int_{B_{R}} \mathcal{H}\left(x, \frac{|u - \lambda|}{R - \varrho}\right)dx,
\end{align}
and
\begin{align}\label{Cacci-ineq}
\int_{B_{\varrho}} \mathcal{H}(x,|\nabla (u - \lambda)_{\pm}|) dx \le C \int_{B_{R}} \mathcal{H}\left(x, \frac{(u - \lambda)_{\pm}}{R - \varrho}\right)dx,
\end{align}
for all $\lambda \in \mathbb{R}$. Here, $(\cdot)_{+}$ and $(\cdot)_{-}$ denote the positive and negative parts, respectively, given by
$$(u(x)-\lambda)_{+} = \max\{u(x)-\lambda, 0\} \quad \text{and} \quad (u(x)-\lambda)_{-} = \max\{\lambda - u(x), 0\}.$$ 
\end{lemma}
\begin{proof}
For any $\tau$ and $\rho$ such that $0<\varrho \le \tau <\rho \le R$, let $\eta \in C_c^{\infty}(B_{\rho})$ be a smooth cut-off function satisfying
\begin{align}\label{eta-set}
0 \le \eta \le 1, \ \eta \equiv 1 \mbox{ on } B_{\tau}, \mbox{ and } |\nabla \eta| \le \frac{2}{\rho-\tau}.
\end{align}
Taking $w = u - \eta (u-\lambda)$, it is clear to see that 
\begin{align}\label{eta-w-set}
u - w \in W_0^{1,\mathcal{H}}(B_{\rho}) \mbox{ and } \nabla w = (1 - \eta) \nabla u - (u-\lambda) \nabla \eta.
\end{align}
Since $u$ is a local minimizer of $\mathbb{F}$,  \eqref{var-form} yields
\begin{align}\notag
\mathbb{F}(u,\mathrm{supp}(u-w)) \le \mathbb{F}(w,\mathrm{supp}(u-w)),
\end{align}
which implies
\begin{align}\notag
\int_{B_{\rho}} \left(|\nabla u|^p + \sum_{i=1}^N a_i(x) |\nabla u|^{p_i}\right)dx \le \int_{B_{\rho}} \left(|\nabla w|^p + \sum_{i=1}^N a_i(x) |\nabla w|^{p_i}\right)dx.
\end{align}
By \eqref{eta-set} and \eqref{eta-w-set}, one has $|\nabla w| \le (1-\eta)|\nabla u| + 2\left|\frac{u-\lambda}{\rho - \tau}\right|$ in $B_{\rho}$ and $\eta = 1$ in $B_{\tau}$. It follows that
\begin{align}\notag
\int_{B_{\rho}} \left(|\nabla u|^p + \sum_{i=1}^N a_i(x) |\nabla u|^{p_i}\right)dx & \le C_1 \int_{B_{\rho}\setminus B_{\tau}} \left(|\nabla u|^p + \sum_{i=1}^N a_i(x) |\nabla u|^{p_i}\right)dx \\
& \qquad + C\int_{B_{\rho}} \left(\left|\frac{u-\lambda}{\rho - \tau}\right|^p + \sum_{i=1}^N a_i(x) \left|\frac{u-\lambda}{\rho - \tau}\right|^{p_i}\right)dx. \label{Ca-0}
\end{align}
Adding to both sides of \eqref{Ca-0} by the following term
\begin{align}\notag
C_1 \int_{B_{\tau}} \left(|\nabla u|^p + \sum_{i=1}^N a_i(x) |\nabla u|^{p_i}\right)dx,
\end{align}
we obtain that
\begin{align}\notag
\int_{B_{\tau}} \mathcal{H}(x,|\nabla u|) dx & \le \frac{C_1}{1+C_1} \int_{B_{\rho}} \mathcal{H}(x,|\nabla u|) dx \\
& \qquad \qquad +  \sum_{i=0}^N \frac{C_2}{(\rho-\tau)^{p_i}}\int_{B_{\rho}} a_i(x) \left|u-\lambda\right|^{p_i} dx, \label{Giu-1}
\end{align}
where $p_0 := p$ and $a_0 \equiv 1$. \\
Let us consider a strictly increasing sequence $(s_m)_{m\ge0} \subset [\varrho, R]$ defined by 
$$s_0 = \varrho \mbox{ and } s_{m+1}= s_m + \epsilon^m (1-\epsilon) (R-\varrho), \quad m \ge 0,$$ 
where $\epsilon \in (0,1)$ is a constant to be chosen later. For every $m \ge 0$, applying~\eqref{Giu-1} with $\tau = s_m$ and $\rho = s_{m+1}$, one has
\begin{align}\notag
\int_{B_{s_{m}}} \mathcal{H}(x,|\nabla u|) dx \le \sigma \int_{B_{s_{m+1}}} \mathcal{H}(x,|\nabla u|) dx + \sum_{i=0}^N \frac{1}{\epsilon^{m p_i}(1-\epsilon)^{p_i}} \frac{b_i}{(R-\varrho)^{p_i}}.
\end{align}  
Setting $\pi =  \max_{i \in I_N} p_i$, it follows that
\begin{align}\notag
\int_{B_{s_{m}}} \mathcal{H}(x,|\nabla u|) dx \le \sigma \int_{B_{s_{m+1}}} \mathcal{H}(x,|\nabla u|) dx + \left[\frac{1}{(1-\epsilon)^{\pi}}  \sum_{i=0}^N \frac{b_i}{(R-\varrho)^{p_i}}\right] \epsilon^{-m \pi},
\end{align}
for every $m \ge 0$, where $\sigma$ and $b_i$ are explicitly given by
\begin{align*}
\sigma = \frac{C_1}{1+C_1},  \text{ and }  b_i = C_2\int_{B_{\rho}} a_i(x) \left|u-\lambda\right|^{p_i} dx,
\end{align*}
Iterating this inequality $k$ times ($k \ge 1$), we observe
\begin{align}\label{Giu-3}
\int_{B_{s_{0}}} \mathcal{H}(x,|\nabla u|) dx & \le \sigma^{k} \int_{B_{s_{k}}} \mathcal{H}(x,|\nabla u|) dx \notag \\
& \qquad \qquad +  \left[\frac{1}{(1-\epsilon)^{\pi}}  \sum_{i=0}^N \frac{b_i}{(R-\varrho)^{p_i}}\right] \sum_{m=0}^{k-1} \big(\sigma \epsilon^{-\pi}\big)^m .
\end{align}
Since $\sigma \in (0,1)$, it is possible to choose $\epsilon \in (0,1)$ sufficiently close to $1$ such that $\sigma\epsilon^{-\pi} < 1$, which ensures that the geometric series $\sum_{m=0}^{\infty} \big(\sigma \epsilon^{-\pi}\big)^m$ converges. Pass to the limit $k\to \infty$ in \eqref{Giu-3}, we obtain
\begin{align}\notag
\int_{B_{\varrho}} \mathcal{H}(x,|\nabla u|) dx & \le C \sum_{i=0}^N \frac{C_2}{(R-\varrho)^{p_i}}\int_{B_{\rho}} a_i(x) \left|u-\lambda\right|^{p_i} dx \\
& \le C \int_{B_{R}} \mathcal{H}\left(x, \frac{|u - \lambda|}{R - \varrho}\right)dx. \notag
\end{align}
This completes the proof of \eqref{Cacci-ineq-0}. A similar argument applies to the case of $(u-\lambda)_{\pm}$, by replacing $w$ in \eqref{eta-w-set} by $w = u - \eta (u-\lambda)_{\pm}$, we analogously obtain \eqref{Cacci-ineq}.
\end{proof}

\subsection{Proof of Theorem \ref{theo:higher_int}}

Once having the previous lemmas at hand, we are in a position to prove Theorem \ref{theo:higher_int}. 

\begin{proof}[Proof of Theorem \ref{theo:higher_int}]
Let $B_{2r} \Subset \Omega$ be a given ball, and we denote $B = B_r$ and $2B = B_{2r}$. Thanks to the Caccioppoli inequality \eqref{Cacci-ineq-0} established in Lemma \ref{lem-Cacci}, there exists a constant $C = C(\mathtt{data}) > 0$ such that
\begin{align}\notag
\fint_{B} \mathcal{H}(x,|\nabla u|) dx \le C \fint_{2B} \mathcal{H}\left(x, \frac{|u - (u)_{2B}|}{r}\right)dx.
\end{align}
Applying the Sobolev-Poincar\'e inequality (Lemma \ref{lem-Sob-Poin}), one can find a constant $\theta > 1$ such that
\begin{align}\notag
\fint_{2B} \mathcal{H}\left(x, \frac{|u - (u)_{2B}|}{r}\right)dx \le C\left[\fint_{2B} \big[\mathcal{H}(x,|\nabla u|)\big]^{\frac{1}{\theta}}dx\right]^{\theta}.
\end{align}
Combining these estimates, we arrive at the following reverse H\"older inequality:
\begin{align}\notag
\fint_{B} \mathcal{H}(x,|\nabla u|) dx \le C \left[\fint_{2B} \big[\mathcal{H}(x,|\nabla u|)\big]^{\frac{1}{\theta}}dx\right]^{\theta}.
\end{align}
Furthermore, since $\theta > 1$, we have $\frac{1}{\theta} < 1$, applying Lemma \ref{lem:Re-Holder}, it yields the existence of a constant $\delta_0=\delta_0(\mathtt{data})>0$ such that $\mathcal{H}(x, |\nabla u|) \in L^{1+\delta_0}_{loc}(\Omega)$. This implies the higher integrability result \eqref{eq:gehring}. 
\end{proof}

\section{Local bounded minimizers}
\label{sec:Linf}

In this section, we provide the proof of Theorem \ref{theo-Linf},  whose main objective is to establish the local boundedness of minimizers of the multi-phase functional $\mathbb{F}$. The argument follows the classical De Giorgi iteration scheme adapted to the multi-phase structure of $\mathcal{H}$. The key ingredients of our approach are the Caccioppoli inequality for level truncations, the Sobolev-Poincar\'e inequality, and the Jensen-type inequality \eqref{phi-phi} established in Lemma \ref{lem-Jensen}. The desired local boundedness then follows via a standard De Giorgi-type iteration lemma, which we recall below.

\subsection{De Giorgi-type iteration}

\begin{lemma}[De Giorgi-type iteration]
\label{lem:DeGiorgi}
Let $\mu>0$, $\alpha>0$, $\beta>0$, and $K>1$. Assume that $\Psi: [0, +\infty) \to [0, +\infty)$ is a strictly increasing function satisfying
\begin{align}
\label{assump-beta}
\Psi(\lambda t) \le \lambda^{\beta} \Psi(t), \quad \text{for all } t \ge 0 \text{ and } \lambda \le 1.
\end{align}
Let $(Y_m)_{m \ge 0}$ be a sequence of non-negative real numbers such that
\begin{align}\label{assump-Ym}
\Psi(Y_0) \le \mu^{-\frac{1}{\alpha}} K^{-\frac{1}{\alpha} - \frac{1}{\alpha^2\beta}},
 \mbox{ and } \ Y_m \le \mu  K^m Y_{m-1} \left[\Psi\big(Y_{m-1}\big)\right]^{\alpha}, 
\end{align}
for every $m \ge 1$. Then, the sequence $(Y_m)_{m \ge 0}$ satisfies the following geometric decay estimate
\begin{align}
\label{Ym-ansatz}
Y_m \le K^{-\frac{m}{\alpha\beta}} Y_0, \quad \text{for all } m \ge 0.
\end{align}
In particular, $\lim_{m \to \infty} Y_m = 0$.
\end{lemma}
\begin{proof}
We shall prove by induction on $m$. Clearly, the base case $m = 0$ holds trivially. Now, assume that the induction hypothesis \eqref{Ym-ansatz} is valid up to step $m-1$ for some $m \ge 1$, which means
$Y_{m-1} \le K^{-\frac{m-1}{\alpha\beta}} Y_0$. Thanks to \eqref{assump-beta}, it yields
$$\Psi(Y_{m-1}) \le \Psi\left(K^{-\frac{m-1}{\alpha\beta}}Y_0\right) \le K^{-\frac{m-1}{\alpha}} \Psi(Y_0).$$
Substituting these bounds into the latter relation of the sequence in \eqref{assump-Ym}, we obtain
\begin{align*}
Y_m & \le \mu K^m \left[K^{-\frac{m-1}{\alpha\beta}} Y_0\right] \left[K^{-\frac{m-1}{\alpha}} \Psi(Y_0)\right]^{\alpha} \\
& = \left[K^{-\frac{m}{\alpha\beta}} Y_0\right] \left[\mu K^{1+\frac{1}{\alpha\beta}} \left[\Psi(Y_0)\right]^{\alpha}\right].
\end{align*}
From the first assumption in \eqref{assump-Ym}, we deduce that $\mu K^{1+\frac{1}{\alpha\beta}} [\Psi(Y_0)]^{\alpha} \le 1$. Consequently,
$$ Y_m \le K^{-\frac{m}{\alpha\beta}} Y_0. $$
This finishes the induction steps, and the proof is complete.
\end{proof}

\subsection{Proof of Theorem \ref{theo-Linf}}

\begin{proof}[Proof of Theorem \ref{theo-Linf}]
For each $m \in \mathbb{N}$, let us first introduce the sequences of radii $r_m>0$, concentric balls $B_m$, and truncation levels $\lambda_m$ defined by
\begin{align}\notag
r_m = r \left(1+\frac{1}{2^m}\right), \quad B_m = B_{r_m}, \quad \lambda_m = L\left(1-\frac{1}{2^m}\right), 
\end{align}
where $L>0$ is a constant to be chosen later. By this setting, we observe that
\begin{align*}
r_{m-1} - r_m = \frac{r}{2^m}, \quad B_r \subset B_m \subset B_{m-1} \subset B_0 = B_{2r}, \quad \text{and} \quad \lambda_m - \lambda_{m-1} = \frac{L}{2^m}, 
\end{align*}
for every $m \ge 1$. On the other hand, we also introduce the sequence $(U_m)_{m \ge 0}$ given by
\begin{align}\notag
U_m = \fint_{B_{m}} \mathcal{H}\left(x,\frac{v_m}{r_m}\right)dx,
\end{align}
where the truncated functions $v_m$ and the corresponding level-set $S_m$ are defined by 
\begin{align}\notag
v_m = (u - (u)_{B} - \lambda_m)_{+} \quad \text{and} \quad S_m = \left\{x \in B_{m}: v_m(x) > 0\right\}.
\end{align}
Thanks to H\"older's inequality, with $\theta>1$ specified from \eqref{SP-ineq}, we infer that
\begin{align}\label{eq:holder_de}
U_m = \frac{|S_{m}|}{|B_{m}|} \fint_{S_{m}} \mathcal{H}\left(x,\frac{v_m}{r_m}\right)dx  \le \left[\fint_{B_{m}} \mathcal{H}^\theta \left(x,\frac{v_m}{r_m}\right)  dx \right]^{\frac{1}{\theta}} \left[\frac{|S_{m}|}{|B_{m}|}\right]^{1 - \frac{1}{\theta}}.
\end{align}
Let us now apply Sobolev-Poincar\'e's inequality \eqref{SP-ineq} with the function $\eta v_m \in W_0^{1,\mathcal{H}}(B_{m-\frac{1}{2}})$, where $B_{m-\frac{1}{2}}:= B_{r_{m-\frac{1}{2}}} = B_{\frac{1}{2}\left(r_{m-1}+r_{m}\right)}$ is a concentric ball and $\eta \in C_c^{\infty}(B_{m-\frac{1}{2}})$  is a standard cut-off function satisfying
\begin{align}\notag
0 \le \eta \le 1, \quad \eta \equiv 1 \text{ on } B_{m}, \quad \text{and} \quad |\nabla \eta| \le \frac{4}{r_{m-1}-r_{m}}.
\end{align}
Taking into account that the ratio of the measures $|B_{m-\frac{1}{2}}|/|B_m|$ is bounded by a positive constant depending only on the dimension $n$, one obtains that
\begin{align}\label{app:SP-1}
\left[\fint_{B_{m}} \mathcal{H}^\theta\left(x, \frac{v_m}{r_{m}}\right)  dx \right]^{\frac{1}{\theta}} & \le C \left[\fint_{B_{m-\frac{1}{2}}} \mathcal{H}^\theta\left(x, \frac{\eta v_m}{r_{m}}\right)  dx \right]^{\frac{1}{\theta}} \notag \\
& \le C \left[\fint_{B_{m-\frac{1}{2}}} \mathcal{H}^\theta\left(x, \frac{\eta v_m}{r_{m-\frac{1}{2}}}\right)  dx \right]^{\frac{1}{\theta}} \notag \\
 & \le C \fint_{B_{m-\frac{1}{2}}} \mathcal{H}(x, |\nabla (\eta v_m)|) dx,
\end{align}
where $C = C(\mathtt{data}) > 0$. Moreover, it is worth noting that 
$$|\nabla (\eta v_m)| \le |\nabla v_m|\eta + |\nabla \eta| v_m \le |\nabla v_m| + 2^{m+2}\frac{v_m}{r},$$
and combining this with the Caccioppoli inequality \eqref{Cacci-ineq}, it follows that
\begin{align}\label{app:SP-2}
\fint_{B_{m-\frac{1}{2}}} \mathcal{H}(x, |\nabla (\eta v_m)|)  dx  & \le C \left[\fint_{B_{m-\frac{1}{2}}} \mathcal{H}(x, |\nabla v_m|)  dx + \fint_{B_{m-\frac{1}{2}}} \mathcal{H}\left(x, 2^{m+2}\frac{v_m}{r}\right)  dx\right] \notag \\
& \le C \left[\fint_{B_{m-1}} \mathcal{H}(x, 2^{m+1}\frac{v_m}{r})  dx + \fint_{B_{m-\frac{1}{2}}} \mathcal{H}\left(x, 2^{m+2}\frac{v_m}{r}\right)  dx\right] \notag \\
& \le C2^{(m+3)p_N} \fint_{B_{m-1}} \mathcal{H}\left(x, \frac{v_{m-1}}{r_{m-1}}\right)  dx.
\end{align}
Combining \eqref{app:SP-1} and \eqref{app:SP-2}, one obtains
\begin{align}\label{app:SP-3}
\left[\fint_{B_{m}} \mathcal{H}^\theta\left(x, \frac{v_m}{r_{m}}\right)  dx \right]^{\frac{1}{\theta}} & \le C 2^{(m+3)p_N}  \fint_{B_{m-1}} \mathcal{H}\left(x, \frac{v_{m-1}}{r_{m-1}}\right)  dx \notag \\
& \le C 2^{(m+3)p_N}  U_{m-1}.
\end{align}
Let us now estimate the measure of the level set $S_m$. For every $x \in S_m$, one can verify that 
$$v_{m-1}(x) > \lambda_{m} - \lambda_{m-1} = \frac{L}{2^m} \chi_{S_m}(x),
$$ 
which implies that
\begin{align}\notag
U_{m-1} & = \frac{1}{|B_{m-1}|} \int_{B_{m-1}} \mathcal{H}\left(x,\frac{v_{m-1}}{r_{m-1}}\right)dx \\
 & \ge \frac{1}{|B_{m-1}|} \int_{S_{m}} \mathcal{H}\left(x,\frac{v_{m-1}}{r_{m-1}}\right)dx  \notag \\
 &\ge \frac{|B_{m}|}{|B_{m-1}|} \fint_{B_{m}} \mathcal{H}\left(x,\frac{L\chi_{S_m}(x)}{2^m r_{m-1}}\right)dx.\notag
\end{align}
Since the volume ratio $|B_m|/|B_{m-1}|$ is bounded from below, it enables us to apply the Jensen-type inequality \eqref{phi-phi} in Lemma \ref{lem-Jensen} to deduce that
\begin{align}\notag
U_{m-1} \ge C \fint_{B_{m}} \mathcal{H}\left(x, \frac{L|S_m|}{2^m r_{m-1}}\right)dx \ge C \Phi_{2B}\left(\frac{L|S_m|}{2^m r_{m-1}}\right).
\end{align}
By Lemma \ref{lem:scale_comparison}, this inequality yields
\begin{align}
|S_m| \le C\Phi_{2B}^{-1}\left(U_{m-1}\right) \frac{2^{m} r_{m-1}}{L} \le C \frac{2^{m} r}{L} \Phi_{2B}^{-1}\left(U_{m-1}\right).\label{app:SP-4}
\end{align}
Substituting \eqref{app:SP-3} and \eqref{app:SP-4} into \eqref{eq:holder_de}, one has
\begin{align}
U_m & \le C 2^{(m+3)p_N}  U_{m-1} \left[\frac{2^{m} r}{L} \Phi_{2B}^{-1}\left(U_{m-1}\right) \right]^{1-\frac{1}{\theta}} \notag \\
& \le C_0  {r}^{1-\frac{1}{\theta}} {L}^{-\left(1-\frac{1}{\theta}\right)} 2^{m\left(p_N + 1-\frac{1}{\theta}\right)}  U_{m-1} \left[\Phi_{2B}^{-1}\left(U_{m-1}\right) \right]^{1-\frac{1}{\theta}}. \notag
\end{align}
Next, we apply Lemma \ref{lem:DeGiorgi} with the parameters $\alpha = 1-\frac{1}{\theta}, \beta = \frac{1}{p_N}$ and
$$\mu = C_0 {r}^{1-\frac{1}{\theta}} {L}^{-\left(1-\frac{1}{\theta}\right)}, \quad K = 2^{p_N + 1-\frac{1}{\theta}}, \quad \Psi = \Phi_{2B}^{-1},$$ 
we can guarantee that $\lim_{m \to \infty} U_m = 0$, provided that the initial value $U_0$ must satisfy the condition
\begin{align}\notag
\Phi_{2B}^{-1}(U_0) & \le \mu^{-\frac{1}{\alpha}} K^{-\frac{1}{\alpha} - \frac{1}{\alpha^2\beta}}= C_0^{-\frac{\theta}{\theta-1}} L r^{-1} 2^{-\left(1+\frac{p_N\theta}{\theta-1}\right)^2}.
\end{align}
To this end, it is sufficient to choose the constant $L>0$ by
$$L = C_0^{\frac{\theta}{\theta-1}} 2^{\left(1+\frac{p_N\theta}{\theta-1}\right)^2}r \Phi_{2B}^{-1}(U_0),$$ 
Therefore, we may conclude that
\begin{align}\notag
\|\big(u - (u)_{2B}\big)_+\|_{L^{\infty}(B)} \le L = C r \Phi_{2B}^{-1}(U_0),
\end{align}
which leads to
\begin{align}\notag
\Phi_{2B}\left(\frac{\|\big(u - (u)_{2B}\big)_+\|_{L^{\infty}(B)}}{r}\right) \le C U_0 \le C \fint_{2B} \mathcal{H}\left(x,\frac{|u - (u)_{2B}|}{r}\right)dx.
\end{align}
We proceed analogously for the function $(u-(u)_{2B})_{-}$ to conclude \eqref{Linf-u}. Finally, inequality \eqref{Linf-u-new} can be established similarly, provided that we replace $v_m$ with the function $w_m:= (u-\lambda_m)_+$ and make the Sobolev-Poincar\'e inequality \eqref{SP-ineq-new} instead of \eqref{SP-ineq}. The proof is now complete.
\end{proof}

\vspace{0.5cm}
We now accomplish the proof of Corollary \ref{coro-Linf}.

\begin{proof}[Proof of Corollary \ref{coro-Linf}]
This corollary is directly derived from Theorem \ref{theo-Linf}. Indeed, according to this theorem, for any ball $B:=B_r$ such that $2B= B_{2r} \Subset \Omega$, there exists a constant $C = C(\mathtt{data})>0$ such that the following estimate holds:
\begin{align}\notag
\fint_{2B}\mathcal{H}\left(x,\frac{\|u - (u)_{2B}\|_{L^{\infty}(B)}}{r}\right)dx \le C \fint_{2B} \mathcal{H}\left(x,\frac{|u - (u)_{2B}|}{r}\right)dx.
\end{align}
On the other hand, the following inequality is obvious
\begin{align*}
\left(\frac{\|u - (u)_{2B}\|_{L^{\infty}(B)}}{r}\right)^p \le \fint_{2B}\mathcal{H}\left(x,\frac{\|u - (u)_{2B}\|_{L^{\infty}(B)}}{r}\right)dx.
\end{align*}
This inequality implies to \eqref{Linf-u-coro}. 
\end{proof}

\section{H\"older continuity of local minimizers}
\label{sec:holder_cont}
\subsection{De Giorgi-type level-set inequality}
\label{sec:levelset}

The current section is devoted to deriving specialized level-set inequalities, which play a key role as a bridge between the boundedness and continuity properties of minimizers. In the context of our multi-phase problem, these estimates are of independent interest, as they allow us to characterize the interaction between different phases and quantify the oscillations, which is fundamental to proving the H\"older continuity of the local minimizers.

\begin{lemma}\label{lem:sum-Ym}
Let $(Y_m)_{m \ge 1}$ be a sequence of non-negative real numbers such that
\begin{align}
\label{hyp:structure}
\sum_{m=1}^{\infty} Y_m \le 1, \quad \text{and} \quad \sum_{m=k+1}^{\infty} Y_m \le \mu Y_k^{\alpha}, \quad \text{for all } k \ge 1.
\end{align}
for some $\mu > 0$ and $\alpha \in \left(0, \frac{1}{2}\right)$. Then, there exists a positive constant $C = C(\alpha) > 0$ such that the following  decay estimate holds
\begin{equation}\label{sum-Ym}
\sum_{m=k}^{\infty} Y_m \le C \mu^{\frac{1}{1-\alpha}} k^{-\frac{\alpha}{1-\alpha}}, \quad \text{for all } k \ge 1.
\end{equation}
\end{lemma}
\begin{proof}
For each $k \ge 1$, we introduce a sequence denoted by $Z_k := \sum_{m=k}^{\infty} Y_m$. It is readily seen that $(Z_k)_{k \ge 1}$ is a non-negative, monotonically decreasing sequence that converges to $0$ as $k \to \infty$, with $Y_k = Z_k - Z_{k+1}$. In terms of $Z_k$, the structural assumption \eqref{hyp:structure} becomes
\begin{align}\label{eq:ind_step_ineq}
Z_k \ge Z_{k+1} + \mu_1 Z_{k+1}^{\frac{1}{\alpha}} \quad \text{for all } k \ge 1,
\end{align}
where $\mu_1 := \mu^{-\frac{1}{\alpha}} > 0$. By induction, we prove that there exists a sufficiently large constant $C_0 \ge 1$ such that
\begin{align}\label{eq:induction_goal}
Z_k \le C_0 k^{-\gamma} \quad \text{for all } \  k \ge 1, \quad \text{where } \gamma := \frac{\alpha}{1-\alpha} \in (0,1).
\end{align}
For the base case $k=1$, since $Z_1 \le 1$, \eqref{eq:induction_goal} holds true by choosing $C_0 \ge Z_1$. Now, assume that the induction hypothesis \eqref{eq:induction_goal} holds for some $k \ge 1$. By contradiction, assume that $Z_{k+1} > C_0 (k+1)^{-\gamma}$. Since the mapping $t \mapsto t + \mu_1 t^{1/\alpha}$ is strictly increasing for $t > 0$, by substituting this lower bound into \eqref{eq:ind_step_ineq} and noting that $\frac{\gamma}{\alpha} = \gamma + 1$, we infer that
\begin{align*}
Z_k &> C_0(k+1)^{-\gamma} + \mu_1 \left[ C_0(k+1)^{-\gamma} \right]^{\frac{\gamma + 1}{\gamma}} = C_0 (k+1)^{-\gamma} \left[ 1 + \mu_1 C_0^{\frac{1}{\gamma}} (k+1)^{-1} \right].
\end{align*}
To reach a contradiction with the hypothesis $Z_k \le C_0 k^{-\gamma}$, it suffices to show that
\begin{align*}
C_0 (k+1)^{-\gamma} \left[ 1 + \frac{\mu_1 C_0^{\frac{1}{\gamma}}}{k+1} \right] \ge C_0 k^{-\gamma} \iff 1 + \frac{\mu_1 C_0^{\frac{1}{\gamma}}}{k+1} \ge \left( 1 + \frac{1}{k} \right)^\gamma.
\end{align*}
Since $\gamma \in (0,1)$, by employing Bernoulli's inequality, we observe that $\left( 1 + \frac{1}{k} \right)^\gamma \le 1 + \frac{\gamma}{k}$ holds for all $k \ge 1$. Thus, to establish the desired inequality, it suffices to ensure that
\begin{align*}
\frac{\mu_1 C_0^{\frac{1}{\gamma}}}{k+1} \ge \frac{\gamma}{k} \iff C_0 \ge \left[\frac{\gamma(k+1)}{\mu_1 k}\right]^{\gamma}.
\end{align*}
At this stage, we can choose the constant $C_0$ sufficiently large by setting
$$C_0 =  (2\gamma)^{\gamma} \mu^{\frac{1}{1-\alpha}} =  (2\gamma)^{\gamma} \mu^{\gamma+1} = (2\gamma)^{\gamma} \mu_1^{-\gamma} \ge \left[\frac{(k+1)\gamma}{k\mu_1}\right]^{\gamma}.$$ 
This choice directly ensures that $1 + \displaystyle{\frac{\mu_1 C_0^{\frac{1}{\gamma}}}{k+1}} \ge \left( 1 + \frac{1}{k} \right)^\gamma$, which implies $Z_k > C_0 k^{-\gamma}$. This contradicts the induction hypothesis, where $Z_k \le C_0 k^{-\gamma}$. Consequently, we obtain $Z_{k+1} \le C_0 (k+1)^{-\gamma}$, and the induction step is complete. Finally, noting that $\sum_{m=k}^{\infty} Y_m = Z_k \le C_0 k^{-\gamma}$, we conclude that \eqref{sum-Ym} holds with $C = (2\gamma)^{\gamma} > 0$. The proof is now complete.
\end{proof}

\begin{lemma}\label{lem-LV-1}
Let $u \in W^{1,\mathcal{H}}(\Omega)$ be a local minimizer of the multi-phase functional $\mathbb{F}$ defined in \eqref{eq-main} under the assumption \eqref{eq:ai}. Let $B_{2r} \Subset \Omega$ be a concentric ball of radius $2r > 0$, and denote $B := B_r$. Assume that $u \ge 0$ in $2B$ and there exist constants $\nu > 0$ and $\sigma \in (0, 1)$ such that 
\begin{align}\label{asump-Hol}
|\{x \in B: \, u(x) \ge  \nu\}| \ge \sigma |B|,
\end{align}
Then, for every $\tau>0$, there exists $\varepsilon=\varepsilon(\text{data}, \sigma, \tau) \in (0,1)$ such that
\begin{align}\label{level-set}
|\{x \in B: \ u(x) < \varepsilon \nu\}| \le \tau |B|.
\end{align}
\end{lemma}
\begin{proof}
For each $m \in \mathbb{N}$, let us define $\lambda_m = 2^{1-m}\nu$ and consider the following level sets
\begin{align}\notag
\mathcal{D}_m = & \left\{x \in B: \, u(x) \ge \lambda_m\right\}, \ \mathcal{E}_m = \left\{x \in B: \, u(x) < \lambda_{m+1}\right\}, 
\end{align}
and 
\begin{align}\notag
\mathcal{F}_m = \{x \in B:  \, \lambda_{m+1} \le u(x) < \lambda_m\}.
\end{align}
Since the sequence $(\lambda_m)_{m \ge 1}$ is strictly decreasing, the initial assumption \eqref{asump-Hol} directly implies that
\begin{align}\label{Hol-6}
\sigma |B| \le |\{x \in B: \, u(x) \ge  \lambda_1\}| \le |\{x \in B: \, u(x) \ge  \lambda_m\}| = |\mathcal{D}_m|,
\end{align}
for all $m \ge 1$.\\
Next, let us define a new cut-off function $v_m \in W^{1, \mathcal{H}}(B_r)$ given by
\begin{align}\notag
v_m(x) = \begin{cases} \lambda_m - \lambda_{m+1}, \quad & x \in \mathcal{D}_m, \\
u(x) - \lambda_{m+1}, \quad & x \in \mathcal{F}_m, \\
0, \quad & x \in \mathcal{E}_m. \end{cases}
\end{align}
By this definition of $v_m$, one observes that $v_m = \lambda_m - \lambda_{m+1}$ on $\mathcal{D}_m$ and $v_m = 0$ on $\mathcal{E}_m$. Then, it follows that
\begin{align}\label{Hol-5}
(\lambda_m - \lambda_{m+1}) |\mathcal{D}_m|^{\frac{n-1}{n}} = \left( \int_{\mathcal{D}_m} |v_m|^{\frac{n}{n-1}} dx \right)^{\frac{n-1}{n}} \le \left( \int_{B} |v_m|^{\frac{n}{n-1}} dx \right)^{\frac{n-1}{n}}.
\end{align}
On the other hand, according to the classical Sobolev-Poincar\'e's inequality to $v_m$, there exists a constant $C = C(n)>0$ such that
\begin{align}\notag
\left( \int_{B} |v_m|^{\frac{n}{n-1}}  dx \right)^{\frac{n-1}{n}} &\leq \left( \int_{B} |v_m - (v_m)_{B}|^{\frac{n}{n-1}}  dx \right)^{\frac{n-1}{n}} + \left( \int_{B} |(v_m)_{B}|^{\frac{n}{n-1}}  dx \right)^{\frac{n-1}{n}} \notag \\
& \le C(n) \int_{B} |\nabla v_m|  dx + |(v_m)_{B}| |B|^{\frac{n-1}{n}}.\notag
\end{align}
Furthermore, employing H\"older's inequality to estimate $|(v_m)_{B_r}|$, and noting that $v_m = 0$ on $\mathcal{E}_m$, we arrive at
\begin{align}\notag
|(v_m)_{B}| & = \frac{1}{|B|} \left| \int_{B \setminus \mathcal{E}_m} v_m  dx \right|  \leq \frac{1}{|B|} \left( \int_{B} |v_m|^{\frac{n}{n-1}}  dx \right)^{\frac{n-1}{n}} |B \setminus \mathcal{E}_m|^{\frac{1}{n}}. 
\end{align}
Combining the two above inequalities, we obtain
\begin{align}\notag
\left( \int_{B} |v_m|^{\frac{n}{n-1}}  dx \right)^{\frac{n-1}{n}} & \le C(n) \int_{B} |\nabla v_m| dx + \left(\frac{|B \setminus \mathcal{E}_m|}{|B|}\right)^{\frac{1}{n}} \left( \int_{B} |v_m|^{\frac{n}{n-1}}  dx \right)^{\frac{n-1}{n}},
\end{align}
which yields
\begin{align}
\left[ 1 - \left( \frac{|B \setminus \mathcal{E}_m|}{|B|} \right)^{\frac{1}{n}} \right] \left( \int_B |v_m|^{\frac{n}{n-1}}  dx \right)^{\frac{n-1}{n}} \leq C(n) \int_B |\nabla v_m| dx. \label{Hol-3}
\end{align}
Using the following fundamental inequality
\begin{align}\notag 
1 - (1-t)^{\frac{1}{n}} \geq \frac{t}{n}, \quad \mbox{ for all } t \in (0,1),
\end{align}
it allows us to estimate the factor on the left-hand side as follows
\begin{align}\notag
1 - \left( \frac{|B \setminus \mathcal{E}_m|}{|B|} \right)^{\frac{1}{n}} = 1 - \left( 1 - \frac{|\mathcal{E}_m|}{|B|} \right)^{\frac{1}{n}} \geq \frac{1}{n}\cdot \frac{|\mathcal{E}_m|}{|B|}.
\end{align}
Substituting this estimate into \eqref{Hol-3}, we deduce that
\begin{align}
\left( \int_B |v_m|^{\frac{n}{n-1}}  dx \right)^{\frac{n-1}{n}} \leq C(n) \frac{|B|}{|\mathcal{E}_m|} \int_B |\nabla v_m|  dx \le C \frac{|B|}{|\mathcal{E}_m|} \int_{\mathcal{F}_m} |\nabla u|  dx. \label{Hol-4}
\end{align}
Combining \eqref{Hol-6}, \eqref{Hol-5} together with \eqref{Hol-4}, one obtains the following estimate for the measure of the level set
\begin{align}
\frac{|\mathcal{E}_m|}{|B|} & \leq \frac{C}{\lambda_m - \lambda_{m+1}} \sigma^{-\frac{n-1}{n}} |B|^{-\frac{n-1}{n}} \int_{\mathcal{F}_m} |\nabla u|  dx \notag \\
 & \le \frac{C}{\lambda_m - \lambda_{m+1}} \sigma^{-\frac{n-1}{n}} |B|^{-\frac{n-1}{n}} \left(\int_B \chi_{\mathcal{F}_m} |\nabla u|^{\theta}  dx\right)^{\frac{1}{\theta}} |\mathcal{F}_m|^{1-\frac{1}{\theta}} \notag \\
 & \le \frac{Cr}{\lambda_m - \lambda_{m+1}} \sigma^{-\frac{n-1}{n}} \left(\fint_B \chi_{\mathcal{F}_m} |\nabla u|^{\theta} dx\right)^{\frac{1}{\theta}} \left(\frac{|\mathcal{F}_m|}{|B|}\right)^{1-\frac{1}{\theta}}, \label{Hol-7}
\end{align}
where we recall that $|B|^{1/n} = C(n)r$. On the other hand, employing Jensen's inequality and the Caccioppoli inequality \eqref{Cacci-ineq}, it follows that
\begin{align}
\fint_{2B} \mathcal{H}\left(x, \left(\fint_{2B} \chi_{\mathcal{F}_m} |\nabla u|^{\theta} dy\right)^{\frac{1}{\theta}}\right)dx &\le C\fint_{2B} \mathcal{H}\left(x,\chi_{\mathcal{F}_m} |\nabla u| \right)dx \notag \\
&\le C\fint_{B} \mathcal{H}\left(x,|\nabla (\lambda_m - u)_{+}| \right)dx \notag \\
& \le C\fint_{2B} \mathcal{H}\left(x, \frac{(\lambda_m - u)_{+}}{r} \right)dx \notag \\
& \le C\fint_{2B} \mathcal{H}\left(x, \frac{\lambda_m}{r} \right)dx. \label{Hol-2}
\end{align}
We note that the last inequality in \eqref{Hol-2} holds due to the non-negativity of $u$ in $2B$, which implies that $(\lambda_m - u)_+ \le \lambda_m$. Consequently, invoking Lemma \ref{lem:scale_comparison}, we readily deduce from \eqref{Hol-2} that
\begin{align}\notag %\label{Hol-1}
\left(\fint_{B} \chi_{\mathcal{F}_m} |\nabla u|^{\theta} dx \right)^{\frac{1}{\theta}} \le C \frac{\lambda_m}{r}.
\end{align}
Substituting this estimate into \eqref{Hol-7} and noting the cancellation of $r$, one gets
\begin{align}
\frac{|\mathcal{E}_m|}{|B|} & \le \frac{C \lambda_m}{\lambda_m - \lambda_{m+1}} \sigma^{-\frac{n-1}{n}} \left(\frac{|\mathcal{F}_m|}{|B|}\right)^{1-\frac{1}{\theta}} = 2C \sigma^{-\frac{n-1}{n}} \left(\frac{|\mathcal{F}_m|}{|B|}\right)^{1-\frac{1}{\theta}}.\notag %\label{Hol-7_2}
\end{align}
Due to the facts that $\mathcal{E}_k = \bigcup_{m=k+1}^{\infty} \mathcal{F}_m$ and $\bigcup_{m=1}^{\infty} \mathcal{F}_m \subset B$, one has
\begin{align}\notag
\sum_{m=1}^{\infty} \frac{|\mathcal{F}_m|}{|B|} \le 1, \mbox{ and } \sum_{m=k+1}^{\infty}\frac{|\mathcal{F}_m|}{|B|} \le 2C \sigma^{-\frac{n-1}{n}} \left(\frac{|\mathcal{F}_k|}{|B|}\right)^{1-\frac{1}{\theta}}, \ k \ge 1.
\end{align}
Applying Lemma \ref{lem:sum-Ym} with $Y_m = \frac{|\mathcal{F}_m|}{|B|}$, $\alpha = 1-\frac{1}{\theta}$ and $\mu = 2C \sigma^{-\frac{n-1}{n}}$, we obtain
\begin{align}\notag
\frac{|\mathcal{E}_k|}{|B|} \le \sum_{m=k}^{\infty} \frac{|\mathcal{F}_m|}{|B|} \le C \sigma^{-\frac{(n-1)\theta}{n}} k^{1-\theta}, \ \mbox{ for all } k \ge 1.
\end{align}
It is worth noting that $\theta$ can be chosen sufficiently close to 1 such that $\alpha \in \left(0,\frac{1}{2}\right)$ upon applying Lemma \ref{lem:sum-Ym}. As a consequence of this inequality, for any $\tau>0$, it is possible to choose $k$ and $\varepsilon$ such that
\begin{align*}
C \sigma^{-\frac{(n-1)\theta}{n}} k^{1-\theta} < \tau, \quad \text{ and } \quad 0<\varepsilon < \lambda_{k+1},
\end{align*}
and we immediately ensure the set inclusion
\begin{align*}
\frac{|\{x \in B: \, u(x) < \varepsilon \nu\}|}{|B|} \le \frac{|\mathcal{E}_k|}{|B|} < \tau. 
\end{align*}
This allows us to conclude the inequality \eqref{level-set}.
\end{proof}

\subsection{Proof of Theorem \ref{theo:Holder-cont}}

In this section, we shall complete the proof of the main result stated in Theorem \ref{theo:Holder-cont}. 

\begin{proof}[Proof of Theorem \ref{theo:Holder-cont}]
To prove Theorem \ref{theo:Holder-cont}, it suffices to establish the decay of the oscillation of $u$ in smaller balls. Let $r>0$ and $2B = B_{2r} \Subset \Omega$, we define 
$$\nu := \frac{1}{2} \mathrm{osc}_{2B}u = \frac{1}{2}\left(\sup_{2B}u - \inf_{2B} u\right) \quad \text{ and } \quad \mu := \frac{1}{2}\left(\sup_{2B}u + \inf_{2B} u\right).$$
We consider the following two subsets of $B$ given by
\begin{align}\notag
V = \left\{x \in B: \, u(x) \ge \mu\right\}, \mbox{ and } W = \left\{x \in B: \, u(x) < \mu\right\}.
\end{align}
At this stage, since $|V| + |W| = |B|$, it is worth observing that either $|V| \ge \frac{1}{2}|B|$ or $|W| \ge \frac{1}{2}|B|$. Without loss of generality, we may assume that $|V| \ge \frac{1}{2}|B|$. Note that $u \ge \mu$ is equivalent to $u - \inf_{2B} u \ge \nu$. Thus, we get \begin{align}\notag
|\{x \in B:  u(x) - \inf_{2B} u \ge \nu\}| \ge \frac{1}{2}|B|.
\end{align}
This inequality allows us to apply Lemma \ref{lem-LV-1} to the non-negative function $\tilde{u}:= u - \inf_{2B} u$. Then, for every $\tau \in (0,1/2)$, one can find $\varepsilon = \varepsilon(\text{data}) \in (0,1)$ small enough such that
\begin{align}\label{tau-1}
|\{x \in B: \, u(x) - \inf_{2B} u < \varepsilon \nu\}| \le \tau |B|.
\end{align}
Let us now define $\omega := (\varepsilon\nu + \inf_{2B} u - u)_+$. It can be seen that $\omega$ is a non-negative minimizer of \eqref{eq-main}. Moreover, the level-set estimate \eqref{tau-1} implies that the zero set of $\omega$ in $B$ is sufficiently large. Indeed,
\begin{align}\notag
|\{x \in B: \ \omega(x) = 0\}| = |\{x \in B: \ u(x) - \inf_{2B} u \ge \varepsilon\nu\}| \ge (1-\tau)|B| \ge \frac{1}{2} |B|.
\end{align}
Applying the estimate \eqref{Linf-u-new} in Theorem \ref{theo-Linf} to $\omega$, which yields
\begin{align}
\Phi_{2B}\left(\frac{\sup_{\frac{1}{2} B} \omega}{r}\right) & \le C \fint_{B} \mathcal{H}\left(x,\frac{\omega}{r}\right)dx \notag \\
& \le C \left(\frac{|\{x\in B: \, \omega(x)>0\}|}{|B|}\right)^{1-\frac{1}{\theta}} \left[\fint_{B} \mathcal{H}^{\theta}\left(x,\frac{\omega}{r}\right)dx\right]^{\frac{1}{\theta}} \notag  \\
& \le C \left(\frac{|\{x\in B: \, u(x) - \inf_{2B} u < \varepsilon \nu\}|}{|B|}\right)^{1-\frac{1}{\theta}} \left[\fint_{B} \mathcal{H}^{\theta}\left(x,\frac{\sup_{B} \omega}{r}\right)dx\right]^{\frac{1}{\theta}}. \notag
\end{align}
Thanks to \eqref{tau-1} and the reverse H\"older inequality \eqref{phi-eps} in Lemma \ref{lem:Re-Holder}, it follows that
\begin{align}
\Phi_{2B}\left(\frac{\sup_{\frac{1}{2} B} \omega}{r}\right) & \le C  \tau^{1-\frac{1}{\theta}} \fint_{B} \mathcal{H}\left(x,\frac{\sup_{B} \omega}{r}\right)dx \notag \\
& \le C 2^{p_N} \tau^{1-\frac{1}{\theta}} \Phi_{2B}\left(\frac{\varepsilon \nu}{2r}\right).\notag %\label{tau-2}
\end{align}
By choosing $\tau \in (0, 1/2)$ sufficiently small such that the constant $C 2^{p_N} \tau^{1-\frac{1}{\theta}} \le 1$, we deduce that
\begin{align}\label{tau-20}
\Phi_{2B}\left(\frac{\sup_{\frac{1}{2} B} \omega}{r}\right) \le \Phi_{2B}\left(\frac{\varepsilon \nu}{2r}\right).
\end{align}
Applying Lemma \ref{lem:scale_comparison}, one obtains $\sup_{\frac{1}{2} B} \omega \le \frac{\varepsilon}{2} \nu$, which is equivalent to
\begin{align}\notag
\frac{\varepsilon}{2}\left(\sup_{2B}u - \inf_{2B} u\right) + \inf_{2B} u - \inf_{\frac{1}{2} B} u \le \frac{\varepsilon}{4}\left(\sup_{2B}u - \inf_{2B} u\right),
\end{align}
which leads to the following oscillation decay inequality:
\begin{align}\notag %\label{osc-decay}
\mathrm{osc}_{\frac{1}{2} B} u & = \sup_{\frac{1}{2} B} u - \inf_{\frac{1}{2} B} u \\
& \le \left(\sup_{2B} u - \inf_{2B} u\right) + \left(\inf_{2B} u- \inf_{\frac{1}{2} B} u\right) \notag \\
& \le \left(1-\frac{\varepsilon}{4}\right)\left(\sup_{2B}u - \inf_{2B} u\right) \notag \\
& = \left(1-\frac{\varepsilon}{4}\right) \mathrm{osc}_{2B} u.\notag
\end{align}
It is well-known that such an oscillation decay estimate guarantees $u \in C^{0,\alpha}$, where where the H\"older exponent $\alpha \in (0,1)$ is explicitly given by $\alpha = -\log_4 \left(1-\frac{\varepsilon}{4}\right)$. The proof is now complete.
\end{proof}

\section{Hanack inequality for multi-phase functionals}
\label{sec:Harnack}

This section sets forth the proof of Theorem \ref{theo:Harnack}. It can be seen that \eqref{ineq-Harnack} provides a precise quantitative relation between the local infimum and supremum of non-negative minimizers of our multi-phase functional $\mathbb{F}$. It therefore effectively controls the oscillation of the minimizers at small scales, which is a key ingredient in deriving higher regularity estimates. Nevertheless, establishing this inequality is far from trivial due to the transitions between different growth phases (governed by multiple exponents), and, simultaneously, coupled with the singular/degenerate Muckenhoupt coefficients. 

The proof of the Harnack inequality for non-negative local minimizers proceeds in two main steps: the infimum estimate (weak-type Harnack inequality) and a corresponding supremum estimate, detailed in Lemmas \ref{lem-inf-Har} and \ref{lem-sup-Har}, respectively. The desired full Harnack inequality then directly follows by combining these two complementary estimates. 

\subsection{Weak Harnack inequality}
\label{sec:w_Harnack}

We begin by recalling the classical covering lemma, originally due to Krylov and Safonov \cite{KS1981}. In what follows, we rely on the version formulated in terms of balls instead of cubes. For a detailed proof of this specific formulation, we refer the reader to \cite{KS2001}, and to \cite{BaCoMin2015} for further discussions and related applications.

\begin{lemma}[Covering lemma]
\label{lem-E-Br}
Let $B_{2r} \subset \mathbb{R}^n$ be a concentric ball, $\tilde{\delta} \in (0, 1)$, and $E \subset B_{2r}$ be a measurable subset. We define the set 
$$
\mathbb{I}_{\tilde{\delta}} = \left\{(x, \rho) \in B_{2r} \times \left(0, \frac{4r}{3}\right):  |B_{3\rho}(x) \cap E| \ge \tilde{\delta} |B_\rho(x)|.
\right\}.
$$
Let us also consider the following subset of $B_{2r}$:
\begin{align*}
E_{\tilde{\delta}} := \bigcup_{(x,\rho) \in \mathbb{I}_{\tilde{\delta}}} \big[ B_{3\rho}(x) \cap B_{2r}\big].
\end{align*}
Then, either $E_{\tilde{\delta}} = B_{2r}$ or $|E_{\tilde{\delta}}| \ge \frac{1}{2^n \tilde{\delta}} |E|$ holds. 
\end{lemma}

\begin{lemma}[Infimum-estimate for non-negative local minimizers]
\label{lem-inf-Har}
Let $r>0$ and a concentric ball $B_{10r} \equiv B_{10r}(x_0) \subset \Omega$. If $u \in W_{\mathrm{loc}}^{1,\mathcal{H}}(\Omega)\cap L_{\mathrm{loc}}^{\infty}(\Omega)$ is a non-negative local minimizer of \eqref{eq-main} under the assumption \eqref{eq:ai}, then there exist a positive exponent $q > 0$ and a constant $C \ge 1$, depending only on $\mathtt{data}$, such that the following infimum-estimate holds:
\begin{align}\label{ineq-weak-Harnack}
\left(\fint_{B_{2r}} u^{q}  dx \right)^{\frac{1}{q}} \le C \inf_{B_r} u.
\end{align}
\end{lemma}
\begin{proof}
For every $\lambda>0$, it is possible to find a unique integer $m_{\lambda} \ge 1$ such that
\begin{align*}
m_\lambda := \min \left\{ m \in \mathbb{N} :
\frac{|\{y \in B_{2r} : u(y) > \lambda\}|}{|B_{2r}|}
\le 2^{-(m-1)} \right\}.
\end{align*}
Then,
\begin{align}\label{m-lam}
\left(\frac{1}{2}\right)^{m_{\lambda}}  < \frac{|\{y \in B_{2r}: \, u(y) > \lambda\}|}{|B_{2r}|} \le \left(\frac{1}{2}\right)^{m_{\lambda}-1}.
\end{align}
With $m_{\lambda}$ given in \eqref{m-lam}, we aim to establish the following lower bound for the infimum:
\begin{align}\label{u-m-lambda}
\lambda_0:= \inf_{B_{2r}} u \ge \mu^{m_{\lambda}} \lambda,
\end{align}
where $\mu \in (0,1)$ is a small constant depending only on $\mathtt{data}$, to be specified later. Before verifying \eqref{u-m-lambda}, let us show how this estimate implies the desired weak Harnack inequality \eqref{ineq-weak-Harnack}. Indeed, by setting $\alpha := \log_{1/2} \mu > 0$, a combination of \eqref{m-lam} and \eqref{u-m-lambda} yields
\begin{align*}
\frac{\lambda_0}{\lambda} \ge \mu^{m_{\lambda}} = \left[\left(\frac{1}{2}\right)^{m_{\lambda}}\right]^{\alpha} \ge \left[\frac{1}{2} \frac{|\{y \in B_{2r}: \, u(y) > \lambda\}|}{|B_{2r}|}\right]^{\alpha}.
\end{align*}
Using Fubini's theorem and choosing an exponent $0< q < \frac{1}{\alpha}$, one obtains
\begin{align*}
\left(\fint_{B_{2r}} u^{q}  dx \right)^{\frac{1}{q}} & = \left[\frac{1}{|B_{2r}|}\int_{0}^{\infty} q\lambda^{q-1} |\{y \in B_{2r}: \, u(y) > \lambda\}| \, d\lambda \right]^{\frac{1}{q}} \\
& \le \left[\lambda_0^q + \int_{\lambda_0}^{\infty} q\lambda^{q-1} \frac{|\{y \in B_{2r}: \, u(y) > \lambda\}|}{|B_{2r}|} \, d\lambda \right]^{\frac{1}{q}} \\
& \le \left[\lambda_0^q + 2q\lambda_0^{\frac{1}{\alpha}} \int_{\lambda_0}^{\infty} \lambda^{q-1-\frac{1}{\alpha}} \, d\lambda \right]^{\frac{1}{q}} \\
& \le \left(\frac{1+\alpha q}{1-\alpha q}\right)^{\frac{1}{q}} \lambda_0,
\end{align*}
which allows us to conclude \eqref{ineq-weak-Harnack}, keeping in mind that $\lambda_0 = \inf_{B_{2r}} u \le \inf_{B_r} u$. Thus, the main objective of this entire proof is to establish inequality \eqref{u-m-lambda}. 

Now, we distinguish two possible cases: $m_{\lambda} = 1$ and $m_{\lambda} > 1$. In the first case, whenever $m_{\lambda} = 1$, inequality \eqref{m-lam} immediately yields
\begin{align}\label{B-4r-est}
|\{y \in B_{2r}: \, u(y) > \lambda\}| > \frac{1}{2} |B_{2r}|.
\end{align}
Since $B_{2r} \subset B_{4r}$ and $|B_{4r}| = 2^n |B_{2r}|$, we can extend this estimate to the larger ball $B_{4r}$ as follows:
\begin{align}\notag
|\{y \in B_{4r}: \, u(y) > \lambda\}| \ge |\{y \in B_{2r}: \, u(y) > \lambda\}| > \frac{1}{2} |B_{2r}| \ge \frac{1}{2^{n+1}} |B_{4r}|.
\end{align} 
At this stage, it allows us to apply Lemma \ref{lem-LV-1} with $B = B_{4r}$, $\sigma = \frac{1}{2^{n+1}}$ and $\tau =\frac{1}{2}$, there exists a small constant $\varepsilon_1 = \varepsilon_1(\mathtt{data}) \in (0,1)$ such that
\begin{align}\notag %\label{u2B}
|\{y \in B_{4r}: \, u(y) > \varepsilon_1\lambda\}| \le \frac{1}{2} |B_{4r}|.
\end{align}
We introduce the localized truncation function $\omega := (\varepsilon_1 \lambda - u)_{+}$, this inequality ensures that
\begin{align}\notag 
|\{y \in B_{4r}: \, \omega(y) = 0\}| & = |\{y \in B_{4r}: \, u(y) \ge \varepsilon_1\lambda\}|  > \frac{1}{2} |B_{4r}|.
\end{align}
By the same argument as in the proof of \eqref{tau-20} applied to the non-negative function $\omega$, we obtain $\sup_{B_{2r}} \omega \le \frac{\varepsilon_1}{2} \lambda$. By the definition of $\omega$, it follows that
$\varepsilon_1 \lambda - \inf_{B_{2r}} u \le \frac{\varepsilon_1}{2}  \lambda$, which readily simplifies to
\begin{align}\label{eps3-lambda}
\inf_{B_{2r}} u \ge \frac{\varepsilon_1}{2} \lambda.
\end{align}
Therefore, inequality \eqref{u-m-lambda} holds for the case $m_{\lambda} = 1$ by choosing the structural paramater $\mu \in (0,\frac{\varepsilon_1}{2})$ in \eqref{eps3-lambda}.

For the remaining case $m_{\lambda} > 1$, the index set $J_{\lambda} := \{1,2,...,m_{\lambda}-1\}$ is then not empty. For each $i \in J_{\lambda}$, we define the following two subsets of $B_{2r}$:
\begin{align}\notag
E := \left\{y \in B_{2r}: \, u(y) \ge \mu^{i-1}\lambda\right\}  \mbox{ and }
 E_{\tilde{\delta}} := \bigcup_{\substack{(x,\rho) \in \mathbb{I}_{\tilde{\delta}}}} \big[ B_{3\rho}(x) \cap B_{2r}\big],
\end{align}
where $\tilde{\delta} = \frac{1}{2^{n+1}}$ and the index set $\mathbb{I}_{\tilde{\delta}}$ is determined as in Lemma \ref{lem-E-Br}. Let us now show that
\begin{align}\label{inf-u-eps}
\inf_{B_{3\rho}(x)} u \ge \mu^{i}\lambda, \ \mbox{ for every } (x,\rho) \in \mathbb{I}_{\tilde{\delta}}.
\end{align}
Indeed, for each pair $(x,\rho) \in \mathbb{I}_{\tilde{\delta}}$ and by the definition of $\mathbb{I}_{\tilde{\delta}}$, one has
\begin{align}\notag
|\{y \in B_{3\rho}(x): \, u(y) \ge \mu^{i-1}\lambda\}| & \ge |B_{3\rho}(x) \cap E|  \ge \tilde{\delta} |B_{\rho}(x)| \ge \frac{\tilde{\delta}}{3^n} |B_{3\rho}(x)|.
\end{align}
We recall that \eqref{eps3-lambda} can be proved under condition \eqref{B-4r-est}. In the same spirit, noting that $B_{6\rho}(x) \subset  B_{10r} = B_{10r}(x_0)$, we can also find a constant $\varepsilon_2 = \varepsilon_2(\mathtt{data}) \in (0,1)$ such that 
$$\inf_{B_{3\rho}(x)} u \ge \frac{\varepsilon_2}{2} \mu^{i-1} \lambda = \frac{\varepsilon_2}{2\mu} \mu^{i} \lambda.$$
As a consequence, inequality \eqref{inf-u-eps} holds by choosing $0< \mu \le \frac{1}{2}\min\left\{\varepsilon_1, \varepsilon_2\right\}$. By \eqref{inf-u-eps}, it is obvious that 
\begin{align}\label{E-subset-B}
E_{\tilde{\delta}} \subset \{y \in B_{2r}: \, u(y) \ge \mu^i \lambda\} \subset B_{2r}.
\end{align}
On the other hand, Lemma \ref{lem-E-Br} ensures that for each $i \in J_{\lambda}$, either
\begin{align*}
E_{\tilde{\delta}} = B_{2r} \quad \text{ or } \quad |E_{\tilde{\delta}}| \ge 2 \left|\left\{y \in B_{2r}:  u(y) \ge \mu^{i-1}\lambda\right\}\right|
\end{align*}
holds. Therefore, taking the inclusion \eqref{E-subset-B} into account, we deduce that either
\begin{align}\label{TH1}
\{y \in B_{2r}: \, u(y) \ge \mu^i \lambda\} = B_{2r}
\end{align}
or
\begin{align}\label{TH2}
|\{y \in B_{2r}: \, u(u) \ge \mu^i \lambda\}| \ge 2 \left|\left\{y \in B_{2r}: \, u(y) \ge \mu^{i-1}\lambda\right\}\right|
\end{align}
hold. If there exists some index $i_0 \in J_{\lambda}$ such that \eqref{TH1} is satisfied, it follows that $u(y) \ge \mu^{i_0} \lambda \ge \mu^{m_{\lambda}} \lambda$ for almost everywhere $y \in B_{2r}$. Thus, \eqref{u-m-lambda} holds in this case. Otherwise, i.e., if \eqref{TH2} holds for all $i \in J_{\lambda}$, then by \eqref{m-lam} we have
\begin{align}\notag
|\{y \in B_{2r}: \ u(y) \ge \mu^{m_{\lambda}-1} \lambda\}| & \ge 2 \left|\left\{y \in B_{2r}: \ u(y) \ge \mu^{m_{\lambda}-2}\lambda\right\}\right| \\
& \ge 2^{m_{\lambda}-1} \left|\left\{y \in B_{2r}: \, u(y) \ge \lambda\right\}\right| \notag \\
& \ge \frac{1}{2} |B_{2r}|.\notag
\end{align}
Once again, this inequality still ensures the existence of $\varepsilon_3 = \varepsilon_3(\mathtt{data}) \in (0,1)$ such that 
$$\inf_{B_{2r}} u \ge \frac{\varepsilon_3}{2} \mu^{m_{\lambda}-1} \lambda.$$ 
Hence, by choosing the constant $\mu$ sufficiently small such that 
$$0< \mu \le \frac{1}{2}\min\left\{\varepsilon_1, \varepsilon_2, \varepsilon_3\right\},$$ 
we arrive at $\inf_{B_{2r}} u \ge \mu^{m_{\lambda}} \lambda$, which completes the proof of \eqref{u-m-lambda}. The proof is now complete.
\end{proof}

\subsection{The supremum estimate of Harnack inequality}

The local supremum estimate for non-negative local minimizers is established in the following lemma.
\begin{lemma}[Sup-estimate for non-negative local minimizers]
\label{lem-sup-Har}
Let $r > 0$ and let $B_{2r} = B_{2r}(x_0) \Subset \Omega$. If $u \in W_{\mathrm{loc}}^{1,\mathcal{H}}(\Omega) \cap L_{\mathrm{loc}}^{\infty}(\Omega)$ is a non-negative local minimizer of \eqref{eq-main}, then for every $q \in (0,p)$, there exists a constant $C = C(\mathtt{data}) \ge 1$ such that the following estimate holds:
\begin{align}\label{ineq-sup-Harnack}
\sup_{B_r} u \le C \left( \fint_{B_{2r}} u^{q}  dx \right)^{\frac{1}{q}}.
\end{align}
\end{lemma}
\begin{proof} 
We proceed with an iteration procedure analogous to the proof of Theorem \ref{theo-Linf}, with necessary modifications. For any fixed radii $1 \le \sigma < \varpi \le 2$ and for each $m \in \mathbb{N}_0$, let us define
\begin{align}\notag
r_m = r \left(\sigma+\frac{\varpi -\sigma}{2^m}\right), \quad B_m = B_{r_m}, \quad \lambda_m = L\left(1-\frac{1}{2^m}\right), \\
v_m = (u  - \lambda_m)_{+}, \quad \text{ and } \quad S_m = \left\{x \in B_{m}: \ v_m(x)>0\right\},\notag
\end{align}
where $L > 0$ is a large scaling parameter to be specified later. Thanks to H\"older's inequality together with the Sobolev-Poincar\'e exponent $\theta > 1$ determined from \eqref{SP-ineq}, we obtain:
\begin{align}\label{sup-Har-est-1}
U_m := \fint_{B_{m}} \mathcal{H}\left(x,v_m\right)dx  \le \left[\fint_{B_{m}} \mathcal{H}^\theta \left(x,v_m\right)  dx \right]^{\frac{1}{\theta}} \left[\frac{|S_{m}|}{|B_{m}|}\right]^{1 - \frac{1}{\theta}}.
\end{align}
By following the same lines as in the proof of Theorem \ref{theo-Linf}, one can bound the two terms on the right-hand side of \eqref{sup-Har-est-1} by the similar technique as in \eqref{app:SP-3} and \eqref{app:SP-4}. Indeed, let $\eta \in C_c^{\infty}(B_{r_{m-\frac{1}{2}}})$ with $r_{m-\frac{1}{2}} = \frac{1}{2}\left(r_{m-1}+r_{m}\right)$, be a cut-off function satisfying
\begin{align}\notag
0 \le \eta \le 1, \quad \eta \equiv 1 \text{ on } B_{m}, \quad \text{and} \quad |\nabla \eta| \le \frac{4}{2^{-m}(\varpi -\sigma)r}.
\end{align}
Let $\beta$ be a small scaling number, introduced to guarantee that 
\begin{align}\label{beta-choice}
2^{-m}(\varpi -\sigma)r \le \beta \le r_{m-\frac{1}{2}}.
\end{align}
 Applying Sobolev-Poincar\'e's inequality \eqref{SP-ineq} with the function $\beta\eta v_m \in W_0^{1,\mathcal{H}}(B_{m-\frac{1}{2}})$, one obtains that
\begin{align}\label{estsup:SP-1}
\left[\fint_{B_{m}} \mathcal{H}^\theta\left(x, v_m\right)  dx \right]^{\frac{1}{\theta}} & \le C \left[\fint_{B_{m-\frac{1}{2}}} \mathcal{H}^\theta\left(x, \eta v_m\right)  dx \right]^{\frac{1}{\theta}} \notag \\
& \le C \left(\frac{r_{m-\frac{1}{2}}}{\beta}\right)^{p_N} \left[\fint_{B_{m-\frac{1}{2}}} \mathcal{H}^\theta\left(x, \frac{\beta\eta v_m}{r_{m-\frac{1}{2}}}\right)  dx \right]^{\frac{1}{\theta}} \notag \\
 & \le C \left(\frac{r_{m-\frac{1}{2}}}{\beta}\right)^{p_N} \fint_{B_{m-\frac{1}{2}}} \mathcal{H}(x, |\nabla (\beta\eta v_m)|) dx,
\end{align}
where $C = C(\mathtt{data}) > 0$. Moreover, noting that 
$$2^{-m}(\varpi -\sigma)r = 2\left(r_{m-1} - r_{m-\frac{1}{2}}\right),$$ 
the Caccioppoli inequality \eqref{Cacci-ineq} yields that
\begin{align}\label{estsup:SP-2}
\fint_{B_{m-\frac{1}{2}}} \mathcal{H}(x, |\nabla (\beta\eta v_m)|)  dx  & \le C \fint_{B_{m-\frac{1}{2}}} \mathcal{H}(x, |\nabla (\beta v_m)|)  dx \notag \\
& \qquad \qquad +  C\fint_{B_{m-\frac{1}{2}}} \mathcal{H}\left(x, \frac{4\beta v_m}{2^{-m}(\varpi -\sigma)r}\right)  dx \notag \\
& \le C \fint_{B_{m-1}} \mathcal{H}\left(x, \frac{\beta v_m}{r_{m-1} - r_{m-\frac{1}{2}}}\right)  dx \notag \\
& \qquad \qquad + \fint_{B_{m-\frac{1}{2}}} \mathcal{H}\left(x, \frac{4\beta v_m}{2^{-m}(\varpi -\sigma)r}\right)  dx \notag \\
& \le C \left(\frac{\beta}{2^{-m}(\varpi -\sigma)r}\right)^{p_N} \fint_{B_{m-1}} \mathcal{H}\left(x, v_{m-1}\right)  dx.
\end{align}
We note that the last inequality in \eqref{estsup:SP-2} is derived from Lemma \ref{lem:scaling_H} under constraint \eqref{beta-choice}. Combining \eqref{estsup:SP-1} and \eqref{estsup:SP-2}, one obtains
\begin{align}\label{estsup:SP-3}
\left[\fint_{B_{m}} \mathcal{H}^\theta\left(x, v_m\right)  dx \right]^{\frac{1}{\theta}} & \le C 2^{(m+3)p_N}  \fint_{B_{m-1}} \mathcal{H}\left(x, v_{m-1}\right)  dx \notag \\
& \le C \frac{2^{mp_N}}{(\varpi -\sigma)^{p_N}}  U_{m-1}.
\end{align}
To estimate the level set $S_m$, one can verify that $v_{m-1}(x) > 2^{-m}L \chi_{S_m}(x)$, for every $x \in S_m$. It implies that
\begin{align}\notag
U_{m-1}  & \ge \frac{1}{|B_{m-1}|} \int_{S_{m}} \mathcal{H}\left(x,v_{m-1}\right)dx \ge \frac{|B_{m}|}{|B_{m-1}|} \fint_{B_{m}} \mathcal{H}\left(x,2^{-m}L \chi_{S_m}(x)\right)dx.\notag
\end{align}
It enables us to apply the Jensen-type inequality \eqref{phi-phi} in Lemma \ref{lem-Jensen} to deduce that
\begin{align}\notag
U_{m-1} \ge C \fint_{B_{m}} \mathcal{H}\left(x, 2^{-m}L\frac{|S_m|}{|B_m|}\right)dx \ge C \Phi_{2B}\left(2^{-m}L\frac{|S_m|}{|B_m|}\right).
\end{align}
By Lemma \ref{lem:scale_comparison}, this inequality yields
\begin{align}
\frac{|S_m|}{|B_m|} \le C\Phi_{2B}^{-1}\left(U_{m-1}\right) 2^{m} L^{-1}.\label{estsup:SP-4}
\end{align}
Substituting \eqref{estsup:SP-3} and \eqref{estsup:SP-4} into \eqref{sup-Har-est-1}, this yields
\begin{align}
U_m & \le  \frac{C_0 {L}^{-\left(1-\frac{1}{\theta}\right)}}{(\varpi-\sigma)^{p_N}}  2^{m\left(p_N + 1-\frac{1}{\theta}\right)}  U_{m-1} \left[\Phi_{\varpi B_{r}}^{-1}\left(U_{m-1}\right) \right]^{1-\frac{1}{\theta}}. \notag
\end{align}
Applying Lemma \ref{lem:DeGiorgi}, we deduce that $\lim_{m \to \infty} U_m = 0$ provided we select $L>0$ such that
$$L = \frac{1}{(\varpi-\sigma)^\frac{\theta p_N}{\theta-1}} C_0^{\frac{\theta}{\theta-1}} 2^{\left(1+\frac{p_N\theta}{\theta-1}\right)^2} \Phi_{\varpi B_{r}}^{-1}(U_0).$$
Since $\lim_{m \to \infty} U_m = 0$, one concludes $\sup_{\sigma B_{r}} u \le L$, from which follows that 
\begin{align}
\sup_{\sigma B_{r}} u & \le \frac{C}{(\varpi-\sigma)^\frac{\theta p_N}{\theta-1}} \Phi_{\varpi B_{r}}^{-1}(U_0) \notag \\
& \le \frac{C}{(\varpi-\sigma)^\frac{\theta p_N}{\theta-1}} \Phi_{\varpi B_{r}}^{-1}\left(\fint_{\varpi B_{r}} \mathcal{H}\left(x,u\right) dx\right). \label{H-x-u-1}
\end{align}
For the sake of simplicity, we introduce the notations $p_0 = p$ and $a_0 \equiv 1$. Consequently, $a_i \in A_{p_i}$ for all $i \in \{0,1,\dots,N\}$. Thanks to Remark \ref{rmk:self-improve}, there exists an exponent $\gamma_i > 1$ such that $a_i \in RH_{\gamma_i}$. By setting $\gamma = \min_{0 \le i \le N} \gamma_i > 1$, we ensure that $a_i \in RH_{\gamma}$ for all $i \in \{0,1,\dots,N\}$. This yields the existence of a structural constant $C_{RH} \ge 1$ such that
\begin{align*}
\left(\fint_{\varpi B_{r}} [a_i(x)]^{\gamma} dx \right)^{\frac{1}{\gamma}} \le C_{RH} \fint_{\varpi B_{r}} a_i(x) dx = C_{RH} (a_i)_{\varpi B_{r}}.
\end{align*}
Combining this fact with H\"older's inequality, for every $q \in (0,p)$, one can estimate as follows
\begin{align}
\fint_{\varpi B_{r}} a_i(x) u(x)^{p_i} dx & \le \big(\sup_{\varpi B_{r}} u\big)^{p_i-q} \fint_{\varpi B_{r}} a_i(x) u(x)^q dx \notag\\
& \le \big(\sup_{\varpi B_{r}} u\big)^{p_i-q} \left(\fint_{\varpi B_{r}} [a_i(x)]^{\gamma} dx \right)^{\frac{1}{\gamma}}  \left(\fint_{\varpi B_{r}} [u(x)]^{\frac{q\gamma}{\gamma-1}} dx \right)^{\frac{\gamma-1}{\gamma}} \notag\\
&  \le C_{RH}(a_i)_{\varpi B_{r}}  \big(\sup_{\varpi B_{r}} u\big)^{p_i-q + \frac{q}{\gamma}}  \left(\fint_{\varpi B_{r}} u^{q} dx \right)^{\frac{\gamma-1}{\gamma}} \notag \\
& =  C_{RH}(a_i)_{\varpi B_{r}}  \big(\sup_{\varpi B_{r}} u\big)^{p_i}  \left(\frac{1}{\sup_{\varpi B_{r}} u^{q}}\fint_{\varpi B_{r}} u^{q} dx \right)^{\frac{\gamma-1}{\gamma}}. \label{H-x-u-2}
\end{align}
Summing \eqref{H-x-u-2} for all indices $i \in \{0,1,\dots,N\}$, we arrive at
\begin{align}\label{H-x-u}
\fint_{\varpi B_{r}} \mathcal{H}(x,u) dx & \le C_{RH}\left(\frac{1}{\sup_{\varpi B_{r}} u^{q}} \fint_{\varpi B_{r}} u^q dx\right)^{\frac{\gamma-1}{\gamma}} \sum_{i=0}^N (a_i)_{\varpi B_{r}}  \big(\sup_{\varpi B_{r}} u\big)^{p_i} \notag \\
& \le C_{RH}\left(\frac{1}{\sup_{\varpi B_{r}} u^{q}} \fint_{\varpi B_{r}} u^q dx\right)^{\frac{\gamma-1}{\gamma}} \Phi_{\varpi B_{r}}\big(\sup_{\varpi B_{r}} u\big).
\end{align}
At this stage, thanks to Lemma \ref{lem-coro-phi-1}, we can make use of the structural property 
$$\Phi^{-1}_{\varpi B_r}(s \Phi_{\varpi B_r}(t)) \le s^{\frac{1}{p_N}} t, \ \mbox{ for any } s \in (0,1] \mbox{ and } t \ge 0.$$ 
Substituting \eqref{H-x-u} into \eqref{H-x-u-1} and absorbing the structural constant $C_{RH}$ into a generic constant $C$, it follows that
\begin{align}
\sup_{\sigma B_{r}} u & \le \frac{C}{(\varpi-\sigma)^\frac{\theta p_N}{\theta-1}} \Phi_{\varpi B_{r}}^{-1}\left[\left(\frac{1}{\sup_{\varpi B_{r}} u^{q}} \fint_{\varpi B_{r}} u^q dx\right)^{\frac{\gamma-1}{\gamma}} \Phi_{\varpi B_{r}}\big(\sup_{\varpi B_{r}} u\big)\right] \notag \\
& \le \frac{C}{(\varpi-\sigma)^\frac{\theta p_N}{\theta-1}} \left(\frac{1}{\sup_{\varpi B_{r}} u^{q}} \fint_{\varpi B_{r}} u^q dx\right)^{\frac{\gamma-1}{\gamma p_N}} \big(\sup_{\varpi B_{r}} u\big) \notag \\
& \le \left[\frac{C}{(\varpi-\sigma)^\frac{\theta \gamma p_N^2}{q(\theta-1)(\gamma-1)}} \left(\fint_{\varpi B_{r}} u^q dx\right)^{\frac{1}{q}}\right]^{\frac{q(\gamma-1)}{\gamma p_N}} \big(\sup_{\varpi B_{r}} u\big)^{1-\frac{q(\gamma-1)}{\gamma p_N}}. \label{sup-Har-est-2}
\end{align}
Here, it is worth noting that since $0<q<p<p_N$ and $\gamma > 1$, one has
$$1-\frac{q(\gamma-1)}{\gamma p_N}>0.$$
Applying Young's inequality with $\varepsilon = \frac{1}{2}$ to \eqref{sup-Har-est-2}, we deduce
\begin{align}
\sup_{\sigma B_{r}} u & \le \frac{1}{2} \sup_{\varpi B_{r}} u + \frac{C}{(\varpi -\sigma)^{\delta_0}} \left(\fint_{\varpi B_{r}} u^q dx\right)^{\frac{1}{q}} \notag \\
& \le \frac{1}{2} \sup_{\varpi B_{r}} u + \frac{C}{(\varpi -\sigma)^{\delta_0}} \left(\fint_{B_{2r}} u^q dx\right)^{\frac{1}{q}}, \label{sup-Har-est-3}
\end{align}
where the exponent $\delta_0 > 1$ is given by 
$$\delta_0 = \frac{\theta  \gamma p_N^2}{q(\theta-1)(\gamma-1)} > 1.$$
To deal with \eqref{sup-Har-est-3}, let us introduce the bounded and non-decreasing function
\begin{align*}
G(s) = \sup_{s B_r} u, \quad s \in [1,2].
\end{align*}
We also consider a strictly increasing sequence $(s_m)_{m\ge0} \subset [1, 2]$ defined by 
$$s_0 = 1 \mbox{ and } s_{m+1}= s_m + \epsilon^m (1-\epsilon), \quad m \ge 0,$$ 
where $\epsilon \in (0,1)$ is a small constant to be chosen later. For every integer $m \ge 0$, applying~\eqref{sup-Har-est-3} with $\sigma = s_m$ and $\varpi = s_{m+1}$, one has
\begin{align}\notag
G(s_m) \le \frac{1}{2} G(s_{m+1}) + \frac{C}{\epsilon^{m \delta_0} (1-\epsilon)^{\delta_0}} \left(\fint_{B_{2r}} u^q dx\right)^{\frac{1}{q}}.
\end{align}  
By induction with respect to  $k \ge 1$, we observe
\begin{align}\label{sup-Har-est-4}
G(s_0) & \le 2^{-k} G(s_{k+1}) + \frac{C}{(1-\epsilon)^{\delta_0}} \left(\fint_{B_{2r}} u^q dx\right)^{\frac{1}{q}} \sum_{m=0}^{k} \left(\frac{1}{2\epsilon^{\delta_0}}\right)^m.
\end{align}
Since $\delta_0 > 1$, it is always possible to choose $\epsilon \in (0,1)$ sufficiently close to $1$ such that $2\epsilon^{\delta_0} > 1$. This choice ensures that the geometric series on the right-hand side of \eqref{sup-Har-est-4} is bounded by a finite constant independent of $k$. Since $u \in L_{\mathrm{loc}}^\infty$, $G(s_{k+1}) \le G(2) < \infty$, it enables us to send $k \to \infty$ in \eqref{sup-Har-est-4}, and we readily obtain \eqref{ineq-sup-Harnack}. The proof is now complete.
\end{proof}

\subsection{Proof of Theorem \ref{theo:Harnack}}

\begin{proof}[Proof of Theorem \ref{theo:Harnack}]
The Harnack inequality \eqref{ineq-Harnack} follows directly from the local supremum estimate and the weak Harnack inequality from Lemmas \ref{lem-E-Br} and \ref{lem-sup-Har}, respectively. Specifically, applying Lemma \ref{lem-sup-Har} to the ball $B_r(x_0)$, we obtain
\begin{align*}
\sup_{B_r(x_0)} u \le C_1 \left( \fint_{B_{2r}(x_0)} u^q  dx \right)^{\frac{1}{q}},
\end{align*}
where $q \in (0, p)$. Next, by invoking the weak Harnack inequality established in Lemma \ref{lem-inf-Har}, we have
\begin{equation*}
\left( \fint_{B_{2r}(x_0)} u^q  dx \right)^{\frac{1}{q}} \le C_2 \inf_{B_r(x_0)} u.
\end{equation*}
Combining these two estimates, we conclude
\begin{equation*}
\sup_{B_r(x_0)} u \le C_1 C_2 \inf_{B_r(x_0)} u = C \inf_{B_r(x_0)} u,
\end{equation*}
where $C = C(\mathtt{data}) \ge 1$. This completes the proof of Theorem \ref{theo:Harnack}.
\end{proof}

\section*{Acknowledgement}
This research is funded by Vietnam National Foundation for Science and Technology Development (NAFOSTED), Grant Number: 101.02-2025.03.

\section*{Conflict of Interest}
The authors declared that they have no conflict of interest.

\section*{Declarations}
Data sharing not applicable to this article as no datasets were generated or analysed during the current study.

\newpage
%%%%%%%%%%%%%%%%%%%%%%%%%%%%%%%

\end{document}